\documentclass[final]{siamltex}
\usepackage{}
\usepackage{graphicx}
\usepackage{multirow,bigstrut}
\usepackage{amsmath,amsfonts,amssymb}
\usepackage{mathrsfs}
\usepackage{color}
\usepackage{upgreek}
\usepackage{bm}
\usepackage{verbatim}

\newtheorem{rem}{\hspace{1mm}Remark}[section]
\newtheorem{exam}{\hspace{1mm}Example}[section]

\newcommand{\beq}{\begin{equation}}
\newcommand{\eeq}{\end{equation}}
\newcommand{\bey}{\begin{eqnarray}}
\newcommand{\eey}{\end{eqnarray}}

\title{A Crank-Nicolson Finite Element Method and the Optimal Error Estimates
for the modified Time-dependent Maxwell-Schr\"{o}dinger Equations
\thanks{{This work is supported by National Natural Science Foundation
of China (grant 11571353, 91330202),  and
Project supported by the Funds for Creative Research Group of China
(grant 11321061).}}}
\author{Chupeng Ma\thanks{{Institute of
      Computational Mathematics and Scientific/Engineering Computing,
      Academy of Mathematics and Systems Science, Chinese Academy of
      Sciences, Beijing 100190, China; ({\tt machupeng@lsec.cc.ac.cn}).}}
\and Liqun Cao\thanks{{Corresponding author. LSEC, NCMIS, Institute of
      Computational Mathematics and Scientific/Engineering Computing,
      Academy of Mathematics and Systems Science, Chinese Academy of
      Sciences, Beijing 100190, China; ({\tt clq@lsec.cc.ac.cn}).}}}

\begin{document}

\maketitle

\begin{abstract}
In this paper we consider the initial-boundary value problem for the time-dependent Maxwell-Schr\"{o}dinger equations,
which arises in the interaction between the matter and the electromagnetic field for the semiconductor quantum devices.
A Crank-Nicolson finite element method for solving the problem is presented.
The optimal energy-norm error estimates for the numerical algorithm without any time-step restrictions are derived.
Numerical tests are then carried out to confirm the theoretical results.
\end{abstract}



\section{Introduction}\label{sec-1}

When the characteristic size of the semiconductor device reaches the wavelength of an electron, the quantum effects become
important even dominant and can not be neglected. The accurate electromagnetic theory for the case is quantum electrodynamics
(QED), i.e. the second quantization for the matter and quantization for the electromagnetic field. However, so far it is extremely
difficult even impossible to employ QED to analyze the interaction between the matter and the electromagnetic field for some
complex systems. The semiclassical (or semi-quantum) electromagnetic models are widely used in the semiconductor quantum devices.
The basic idea is that we use the Maxwell's equations for the electromagnetic field while we use the Schr\"{o}dinger equation
of the non-relativistic quantum mechanics for the matter (see \cite{Ez, Sch}). The Maxwell-Schr\"{o}dinger coupled
 system (M-S) is written as follows:
\begin{equation}\label{eq:1-1}
\left\{
\begin{array}{@{}l@{}}
{\displaystyle  \mathrm{i}\hbar\frac{\partial \Psi(\mathbf{x},t)}{\partial t}=
\left\lbrace\frac{1}{2m}\left[\mathrm{i}\hbar\nabla +q\mathbf{A}(\mathbf{x},t)\right]^{2}
 + q \Phi(\mathbf{x},t)+V_{0} \right\rbrace\Psi(\mathbf{x},t),}\\[2mm]
{\displaystyle  \qquad \qquad \qquad \qquad \qquad \qquad \qquad \qquad \qquad \qquad(\mathbf{x},t)\in\Omega\times(0,T),}\\[2mm]
 {\displaystyle -\frac{\partial}{\partial t}\nabla\cdot\big(\epsilon\mathbf{A}(\mathbf{x},t)\big)
 -\nabla\cdot\big(\epsilon\nabla\Phi(\mathbf{x},t)\big) =q |\Psi(\mathbf{x},t)|^{2}, \,\,
 (\mathbf{x},t)\in\Omega\times(0,T),} \\[2mm]
{\displaystyle \epsilon\frac{\partial ^{2}\mathbf{A}(\mathbf{x},t)}{\partial t^{2}}+\nabla\times
\big({\mu}^{-1}\nabla\times \mathbf{A}(\mathbf{x},t)\big)
+\epsilon \frac{\partial (\nabla \Phi(\mathbf{x},t))}{\partial t} =\mathbf{J}_{q}(\mathbf{x},t),}\\[2mm]
 {\displaystyle  \qquad \qquad \qquad \qquad \qquad \qquad \qquad \qquad \qquad \qquad
(\mathbf{x},t)\in \Omega\times(0,T),}\\[2mm]
{\displaystyle \mathbf{J}_q=-\frac{\mathrm{i}q\hbar}{2m}\big(\Psi^{\ast}\nabla{\Psi}-\Psi\nabla{\Psi}^{\ast}\big)-\frac{\vert q\vert^{2}}{m}\vert\Psi\vert^{2}\mathbf{A},} \\[2mm]
{\displaystyle \Psi, \Phi, \mathbf{A} \,\, \mathrm{subject \ to \ the \ appropriate \ initial \ and\ boundary \ conditions}, }
\end{array}
\right.
\end{equation}
where $ \Omega\subset \mathbb{R}^d  $, $ d\geq 2 $ is a bounded Lipschitz polyhedral convex domain,
$ \Psi^{\ast} $ denotes the complex conjugate of $ \Psi $,
$\epsilon$ and $\mu$ respectively denote the electric permittivity and the magnetic permeability of the material and $V_0$ is the constant potential energy.

It is well-known that the solutions of the Maxwell-Schr\"{o}dinger equations (\ref{eq:1-1}) are not unique. In fact,
for any function $\chi : \Omega\times (0,T)\rightarrow \mathbb{R}$, if $(\Psi, \Phi, \mathbf{A}) $ is a solution of (\ref{eq:1-1}), then $(\exp(i\chi)\Psi, \Phi-\partial_{t}\chi, \mathbf{A}+\nabla \chi) $ is also a solution of (\ref{eq:1-1}).
It is often assumed that the further equations can be adjoined to the Maxwell-Schr\"{o}dinger equations by means
of a gauge transformation. In this paper we consider the M-S system (\ref{eq:1-1}) under the temporal gauge (also called Weyl gauge), i.e. $\Phi=0 $.

In this paper we employ the atomic units,
i.e. $\hbar=m=q=1 $. For simplicity, we also assume that $\epsilon=\mu=1$ without
loss of generality. Hence, $\Psi$ and $\mathbf{A}$ satisfy the following Maxwell-Schr\"{o}dinger equations :
\begin{equation}\label{eq:1-2}
\left\{
\begin{array}{@{}l@{}}
{\displaystyle  -\mathrm{i}\frac{\partial \Psi}{\partial t}+
\frac{1}{2}\left(\mathrm{i}\nabla +\mathbf{A}\right)^{2}\Psi
 + V_{0}\Psi = 0 ,\,\, (\mathbf{x},t)\in
\Omega\times(0,T),}\\[2mm]
{\displaystyle \frac{\partial ^{2}\mathbf{A}}{\partial t^{2}}+\nabla\times
(\nabla\times \mathbf{A}) +\frac{\mathrm{i}}{2}\big(\Psi^{*}\nabla{\Psi}-\Psi\nabla{\Psi}^{*}\big)+\vert\Psi\vert^{2}\mathbf{A}=0,
\,\, (\mathbf{x},t)\in \Omega\times(0,T),}\\[2mm]
\end{array}
\right.
\end{equation}
Here we omit the initial and boundary conditions for $\Psi$ and $\mathbf{A}$ temporarily.

Under the temporal gauge, the second equation in (\ref{eq:1-1}) involving the divergence of $\mathbf{A}$  can be rewritten as
\begin{equation}\label{eq:1-2-0}
 -\frac{\partial}{\partial t}\nabla\cdot\mathbf{A}(\mathbf{x},t) = |\Psi(\mathbf{x},t)|^{2}, \,\,(\mathbf{x},t)\in\Omega\times(0,T),
\end{equation}
which can be derived from (\ref{eq:1-2}) if the solutions of (\ref{eq:1-2}) are sufficiently smooth and the initial datas are consistent.

 Integrating with respect to $t$ on the both sides of (\ref{eq:1-2-0}),  we have
\begin{equation}\label{eq:1-3}
{\displaystyle  -\nabla\cdot\mathbf{A}(\mathbf{x},t) + \nabla\cdot\mathbf{A}(\mathbf{x},0) - \int_{0}^{t}\rho(\mathbf{x}, \tau) d\tau = 0 \, ,}
\end{equation}
where $ \rho(\mathbf{x}, t) =|\Psi(\mathbf{x},t)|^{2} $.

For the purpose of theoretical analysis, we take the gradient of (\ref{eq:1-3}), multiply it by a parameter $\gamma > 0$ and add it to the second equation of (\ref{eq:1-2}), to obtain
\begin{equation}\label{eq:1-4}
\left\{
\begin{array}{@{}l@{}}
{\displaystyle  -\mathrm{i}\frac{\partial \Psi}{\partial t}+
\frac{1}{2}\left(\mathrm{i}\nabla +\mathbf{A}\right)^{2}\Psi
 + V_{0}\Psi = 0 ,\,\, (\mathbf{x},t)\in
\Omega\times(0,T),}\\[2mm]
{\displaystyle \frac{\partial ^{2}\mathbf{A}}{\partial t^{2}}+\nabla\times
(\nabla\times \mathbf{A}) -\gamma \nabla(\nabla \cdot \mathbf{A}) +\frac{\mathrm{i}}{2}\big(\Psi^{*}\nabla{\Psi}-\Psi\nabla{\Psi}^{*}\big)+\vert\Psi\vert^{2}\mathbf{A} }\\[2mm]
{\displaystyle \quad\,+\, \gamma \nabla(\nabla \cdot \mathbf{A}(\mathbf{x},0)) - \gamma \nabla \int_{0}^{t}\rho(\mathbf{x}, \tau) d\tau=0,
\,\, (\mathbf{x},t)\in \Omega\times(0,T).}\\[2mm]
\end{array}
\right.
\end{equation}

The parameter $\gamma $ is referred to as the penalty factor. The choice of $\gamma$ depends on how much emphasis one places on the equality (\ref{eq:1-2-0}). In this paper, we keep $\gamma$ fixed. To avoid the difficulty for integro-differential equations, assuming that the change of the density function $\rho(\mathbf{x}, t)$
is smooth with respect to $t$ for all $\mathbf{x} \in \Omega$, we give an approximation of $\int_{0}^{t}\rho(\mathbf{x}, \tau) d\tau $ as follows.

First  denoting by $t_j= jT/M, j=0,1,\cdots,M$, we divide the time interval $[0,T]$ into $M$ subintervals $[0, t_1], (t_1, t_2],\cdots, (t_{M-1}, T]$. For $t \in[0, t_1]$,
$\rho(\mathbf{x}, t)$ is approximated by Taylor expansion and the initial conditions:
\begin{equation*}
\rho(\mathbf{x}, t) \approx \rho(\mathbf{x}, 0)+t\frac{\partial \rho}{\partial t}(\mathbf{x}, 0),\quad \forall x \in \Omega
\end{equation*}
and
\begin{equation*}
 \int_{0}^{t}\rho(\mathbf{x}, \tau) d\tau \approx t\rho(\mathbf{x}, 0)
+ \frac{1}{2}t^{2} \frac{\partial \rho}{\partial t}(\mathbf{x}, 0),\quad \forall x \in \Omega.
\end{equation*}

The computation of  $\frac{\displaystyle \partial \rho}{\displaystyle \partial t}(\mathbf{x}, 0)$ involves $\frac{\displaystyle\partial \Psi}{\displaystyle\partial t}(\mathbf{x}, 0) $,
the time derivative of initial wave function. Here we assume the initial conditions are consistent and so we can obtain
$\frac{\displaystyle \partial \Psi}{\displaystyle\partial t}(\mathbf{x}, 0) $ from Schr\"{o}dinger's equation.
Given an approximation of $\int_{0}^{t}\rho(\mathbf{x}, \tau) d\tau $ in $ [0, t_1]$, we can solve the coupled differential equations (\ref{eq:1-4})
in the subinterval $[0, t_1]$ and integrate the density function to obtain $\int_{0}^{t_1}\rho(\mathbf{x}, \tau) d\tau $.
Then for $t \in (t_1, t_2]$,  $\int_{0}^{t}\rho(\mathbf{x}, \tau) d\tau $ can be calculated as follows:
\begin{equation*}
\begin{array}{@{}l@{}}
{\displaystyle \int_{0}^{t}\rho(\mathbf{x}, \tau) d\tau = \int_{0}^{t_1}\rho(\mathbf{x}, \tau) d\tau
 +\int_{t_1}^{t}\rho(\mathbf{x}, \tau) d\tau} \\[2mm]
{\displaystyle\quad\quad
\quad \approx \int_{0}^{t_1}\rho(\mathbf{x}, \tau) d\tau + (t-t_1)\rho(\mathbf{x}, t_1)+\frac{1}{2}(t-t_1)^{2} \frac{\partial \rho}{\partial t}(\mathbf{x}, t_1), \quad \forall x \in \Omega}.
\end{array}
\end{equation*}

Now we solve the Maxwell-Schr\"{o}dinger equations (\ref{eq:1-4}) in the subinterval $(t_1, t_2]$. Repeating the above procedure,
we can solve the Maxwell-Schr\"{o}dinger equations (\ref{eq:1-4})
in the subinterval $(t_2, t_3],\cdots, (t_{M-1}, T]$ successively. Therefore, we decompose the original system in $[0, T] $ into
M system in $[0, t_1], (t_1, t_2],\cdots, (t_{M-1}, T]$, respectively.
For $t\in (t_{j-1}, t_j], j = 1, 2,\cdots, M$, the Maxwell-Schr\"{o}dinger equations (\ref{eq:1-4}) can be rewritten as follows:
\begin{equation}\label{eq:1-5}
\left\{
\begin{array}{@{}l@{}}
{\displaystyle  -\mathrm{i}\frac{\partial \Psi}{\partial t}+
\frac{1}{2}\left(\mathrm{i}\nabla +\mathbf{A}\right)^{2}\Psi
 + V_{0}\Psi = 0 ,\,\, (\mathbf{x},t)\in
\Omega\times(t_{j-1}, t_j],}\\[2mm]
{\displaystyle \frac{\partial ^{2}\mathbf{A}}{\partial t^{2}}+\nabla\times
(\nabla\times \mathbf{A})-\gamma \nabla(\nabla \cdot \mathbf{A})+\frac{\mathrm{i}}{2}\big(\Psi^{*}\nabla{\Psi}-\Psi\nabla{\Psi}^{*}\big)}\\[2mm]
{\displaystyle \quad\quad+\,\,\vert\Psi\vert^{2}\mathbf{A} =\mathbf{f}_{j}(\mathbf{x},t),\quad \,\, (\mathbf{x},t)\in \Omega\times(t_{j-1}, t_j],}\\[2mm]
\end{array}
\right.
\end{equation}
where $\mathbf{f}_j(\mathbf{x},t)$ is the known function.

\begin{rem}\label{rem1-1}
We can get the modified Maxwell-Schr\"{o}dinger equations (\ref{eq:1-5}) under the assumption that the change of the density function $\rho(\mathbf{x}, t)$
is smooth with respect to $t$ for all $\mathbf{x} \in \Omega$. If the initial wave function $\Psi(\mathbf{x},0)$ is the eigenfunction of the stationary Schr\"{o}dinger equation
and the incoming electromagnetic field is weak and can be considered as a small perturbation to the quantum system, this assumption is reasonable.
The choice of the number $M$ of subintervals depends on the initial wave function, the incoming electromagnetic field and $T$.
\end{rem}

In this paper, we consider the following modified Maxwell-Schr\"{o}dinger equations:
\begin{equation}\label{eq:1-6}
\left\{
\begin{array}{@{}l@{}}
{\displaystyle  -\mathrm{i}\frac{\partial \Psi}{\partial t}+
\frac{1}{2}\left(\mathrm{i}\nabla +\mathbf{A}\right)^{2}\Psi
 + V_{0}\Psi = 0 ,\,\, (\mathbf{x},t)\in
\Omega\times(0,T),}\\[2mm]
{\displaystyle \frac{\partial ^{2}\mathbf{A}}{\partial t^{2}}+\nabla\times
(\nabla\times \mathbf{A}) - \gamma\nabla(\nabla \cdot \mathbf{A}) +\frac{\mathrm{i}}{2}\big(\Psi^{*}\nabla{\Psi}-\Psi\nabla{\Psi}^{*}\big)}\\[2mm]
{\displaystyle \quad\quad+\,\,\vert\Psi\vert^{2}\mathbf{A} =\mathbf{g}(\mathbf{x},t),\quad \,\, (\mathbf{x},t)\in \Omega\times(0,T).}\\[2mm]
\end{array}
\right.
\end{equation}

The boundary conditions are
\begin{equation}\label{eq:1-7}
{\displaystyle \Psi(\mathbf{x},t)=0,\quad \mathbf{A}(\mathbf{x},t)\times\mathbf{n}=0, \quad \nabla \cdot \mathbf{A}(\mathbf{x},t) = 0, \quad (\mathbf{x},t)\in
\partial \Omega\times(0,T),}
\end{equation}
and the initial conditions are
\begin{equation}\label{eq:1-8}
{\displaystyle \Psi(\mathbf{x},0) = \Psi_0(\mathbf{x}),\quad\mathbf{A}(\mathbf{x},0)=\mathbf{A}_{0}(\mathbf{x}),\quad
\mathbf{A}_{t}(\mathbf{x},0)=\mathbf{A}_{1}(\mathbf{x}),}
\end{equation}
where $ \mathbf{A}_{t} $ denotes the derivative of $ \mathbf{A} $ with respect to the time $ t $,
$ \mathbf{n}=(n_1, n_2, n_3) $ is the outward unit normal to the boundary $ \partial \Omega $. We assume that $\nabla \cdot \mathbf{A}_{0} = \nabla \cdot \mathbf{A}_{1}= 0$ on $\partial \Omega$.

\begin{rem}\label{rem1-2}
The boundary condition $
\Psi(x)=0 $ on $\partial \Omega $ implies that the particle is confined in a whole domain $\Omega$. The boundary condition
$  \mathbf{A}(\mathbf{x},t)\times\mathbf{n}=0 $ on $\partial \Omega$  is referred to as the perfect conductive boundary condition.
The boundary condition $\nabla \cdot \mathbf{A}(\mathbf{x},t) = 0$ on $\partial \Omega$ can be deduced from the boundary condition of
$\,\Psi$ and (\ref{eq:1-3}) if the initial conditions $\mathbf{A}_{0}$ and $\mathbf{A}_{1}$ satisfy $\nabla \cdot \mathbf{A}_{0} = \nabla \cdot \mathbf{A}_{1}= 0$
on $\partial \Omega$. As for the determination of the boundary conditions for the vector potential $\mathbf{A}$, we refer to \cite{Weng}.
\end{rem}

\begin{rem}\label{rem1-3}
The existence and uniqueness of the solution for the time-dependent Maxwell-Schr\"{o}dinger equations (\ref{eq:1-2}) have been investigated in \cite{Gin, Guo, Nak, Nak-1, Nak-2, Shi, Wa}. However, the results of the well-posedness of the problem were obtained only for the Cauchy problem in $\mathbf{R}^{d} $, $ d\geq 1 $ instead of the initial-boundary value problem.
To the best of our knowledge, the existence  and uniqueness of the solution for the Maxwell-Schr\"{o}dinger equations
in a bounded domain seem to be open. For the modified equations (\ref{eq:1-6})-(\ref{eq:1-8}), we will investigate the existence of solutions in another paper.
\end{rem}

Many authors have discussed the numerical methods for the time-dependent Maxwell-Schr\"{o}dinger equations.
We recall some important studies about the problem.
Sui and his collaborators \cite{Sui} used the finite-difference time-domain (FDTD) method to
solve the Maxwell-Schr\"{o}dinger equations and to simulate a simple electron tunneling problem.
Pierantoni, Mencarelli and Rozzi \cite{Pi} applied the transmission line matrix method(TLM) to solve
the Maxwell's equations and employed the FDTD method to solve the Schr\"{o}dinger equation, and did the simulation for
a carbon nanotube between two metallic electrodes. Ahmed and Li \cite{Ah-1} used the FDTD method for the Maxwell-Schr\"{o}dinger system
to simulate plasmonics nanodevices.  The numerical studies listed above all include a step where they extract the vector potential $\mathbf{A}$
and the scalar potential $\Phi$ from the electric field $\mathbf{E}$ and the magnetic field $\mathbf{H}$ after solving the Maxwell's equations
involving $\mathbf{E}$ and $\mathbf{H}$. Recently, Ryu \cite{Ryu} employed directly the FDTD scheme to discretize the Maxwell-Schr\"{o}dinger equations (\ref{eq:1-1}) under the Lorentz gauge and to simulate a single electron in an artificial atom excited by an incoming electromagnetic field. Other related studies on this topic have been reported in \cite{Oh, Sa, Tur} and the references therein.

There are few results on the finite element method (FEM) of the Maxwell-Schr\"{o}dinger equations and the convergence analysis. In this paper we will present a Crank-Nicolson finite element method for solving the problem
(\ref{eq:1-6})-(\ref{eq:1-8}), i.e. the finite element method in space and the Crank-Nicolson scheme in time.
Then we will derive the optimal error estimates for the proposed method. Roughly speaking, compared with explicit algorithms such as the FDTD method, our method is more stable and suffers from less restriction in the time step-size since we use the Crank-Nicolson scheme in the time direction. Moreover, our method is more appropriate to deal with materials with discontinuous electromagnetic coefficients than the FDTD method.
our work is motivated by \cite{Mu-2} in which Mu and Huang proposed an alternating Crank-Nicolson method for
the time-dependent Ginzburg-Landau equations. The optimal error estimates were derived under the time step restrictive
conditions $\Delta t \leq O(h^{\frac{11}{12}})$ for the two-dimension model and $\Delta t \leq O(h^{2})$ for the three-dimension model,
where $h$ and $\Delta t$ are the spatial mesh size and the time step, respectively.
The related convergence results associated with the time-dependent Ginzburg-Landau equations can be also given in \cite{Bu-1,Chen-1, Du-2, Gao-1, Mu-1}.
It should be emphasized that although the time-dependent
Ginzburg-Landau model is somehow formally similar to the time-dependent Maxwell-Schr\"{o}dinger system, there exists the essential
difference between them. The former is classified as a parabolic system and the latter is a hyperbolic system.
The main key point in our work is how to avoid using the finite element inverse estimates when dealing with the nonlinear terms.
The new ideas are to derive the energy-norm error estimates for the Schr\"{o}dinger's equation,
and to employ some tricks to eliminate the nonlinear terms both in the Schr\"{o}dinger's equation and in Maxwell's equations, respectively.

The remainder of this paper is organized as follows.
In section~\ref{sec-2}, a Crank-Nicolson scheme with the Galerkin finite element
approximation for the modified Maxwell-Schr\"{o}dinger equations (\ref{eq:1-6})-(\ref{eq:1-8}) is developed.
In section~\ref{sec-3}, the stability estimates are given. The optimal error estimates for the numerical solution without any 
restriction on time step are derived in section~\ref{sec-4}. Finally, the numerical testes are then carried out to confirm the theoretical results.

Throughout this paper, we denote by $C$ a generic positive constant
independent of the mesh size and the time step without distinction.

\section{A Crank-Nicolson Galerkin finite element scheme}\label{sec-2}

In this section, we present a numerical scheme for the modified Maxwell-Schr\"{o}dinger equations (\ref{eq:1-6})-(\ref{eq:1-8}) using Galerkin
finite element method in space and the Crank-Nicolson scheme in time.
To start with, here and afterwards, we assume that $\Omega $ is a bounded Lipschitz polygonal convex domain in $\mathbb{R}^{2}$
(or a bounded Lipschitz polyhedron convex domain
in $\mathbb{R}^{3}$). 

We introduce the following notation. Let $W^{s,p}(\Omega)$ denote the conventional Sobolev spaces of the real-valued functions.
As usual, $W^{s,2}(\Omega)$ and  $W^{s,2}_{0}(\Omega)$ are denoted by  $H^{s}(\Omega)$ and $H^{s}_{0}(\Omega)$ respectively.
We use $\mathcal{W}^{s,p}(\Omega)=\{u+\mathrm{i}v\,|\, u,v \in W^{s,p}(\Omega)\} $
and $\mathcal{H}^{s}(\Omega)=\{u+\mathrm{i}v\,|\, u,v \in H^{s}(\Omega)\}$ with calligraphic letters for Sobolev spaces of the complex-valued functions, respectively.
Furthermore, let $\mathbf{W}^{s,p}(\Omega) =[W^{s,p}(\Omega)]^{d} $ and $\mathbf{H}^{s}(\Omega)=[H^{s}(\Omega)]^{d}$ with bold faced letters be Sobolev spaces of
the vector-valued functions with $d$ components ($d$=2,\,3). $ L^{2}$ inner-products in $H^{s}(\Omega) $, $\mathcal{H}^{s}(\Omega)$
and $\mathbf{H}^{s}(\Omega)$  are denoted by $(\cdot,\cdot )$ without ambiguity.

In particular, we introduce the following subspace of $\mathbf{H}^{1}(\Omega)$:
\begin{equation*}
\mathbf{H}^{1}_{t}(\Omega)=\{\mathbf{A}\,|\,\mathbf{A}\in\mathbf{H}^{1}(\Omega),\,\,\mathbf{A}\times\mathbf{n}=0 \,\, \,{\rm on}\, \,\,\partial\Omega\}
\end{equation*}
The semi-norm on $\mathbf{H}^{1}_{t}(\Omega)$ is defined by
\begin{equation*}
\Vert \mathbf{u} \Vert_{ \mathbf{H}^{1}_{t}(\Omega)} : = \left[\Vert\nabla \cdot \mathbf{u}\Vert^{2}_{\mathbf{L}^{2}(\Omega)} + \Vert\nabla \times \mathbf{u}\Vert^{2}_{\mathbf{L}^{2}(\Omega)
}\right]^{\frac{1}{2}},
\end{equation*}
which is equivalent to the standard $\mathbf{H}^{1}(\Omega)$-norm $\Vert \mathbf{u} \Vert_{ \mathbf{H}^{1}(\Omega)}$, see \cite{Gir}.

To take into account the time-independence,  for a time $T > 0 $ fixed, let $L^p(0,T;X)$ be the Bochner space defined in \cite{Simon} for $p\in[1,\infty]$ and a Banach space $X$.

The weak formulation of the Maxwell-Schr\"{o}dinger system (\ref{eq:1-6})- (\ref{eq:1-8}) can be specified as follows:
given $\mathbf{g} \in {L}^2(0,T;\mathbf{L}^2(\Omega))$, find $(\Psi,\mathbf{A})\in \mathcal{H}_{0}^{1}(\Omega)\times \mathbf{H}_{t}^{1}(\Omega)$ such that $\forall t\in(0,T)$,
\begin{equation}\label{eq:2-1}
\left\{
\begin{array}{@{}l@{}}
{\displaystyle
 -\mathrm{i}(\frac{\partial \Psi}{\partial t},\varphi)+
\frac{1}{2}(\left(\mathrm{i}\nabla +\mathbf{A}\right)\Psi,\left(\mathrm{i}\nabla +\mathbf{A}\right)\varphi)
 + (V_{0}\Psi,\varphi) = 0 ,\quad \forall \varphi\in\mathcal{H}_{0}^{1}(\Omega)}\\[2mm]
{\displaystyle ( \frac{\partial ^{2}\mathbf{A}}{\partial t^{2}},\mathbf{v})+(\nabla\times
 \mathbf{A},\nabla\times\mathbf{v}) + \gamma(\nabla\cdot
 \mathbf{A},\nabla\cdot\mathbf{v})+(\frac{\mathrm{i}}{2}\big(\Psi^{*}\nabla{\Psi}-\Psi\nabla{\Psi}^{*}\big),\mathbf{v}) }\\[2mm]
{\displaystyle\quad \quad  \quad + \,( |\Psi|^{2}\mathbf{A},\mathbf{v})= ( \mathbf{g},\mathbf{v}) ,\qquad \qquad \qquad\qquad\qquad\forall\mathbf{v}\in\mathbf{H}^{1}_{t}(\Omega),}
\end{array}
\right.
\end{equation}
with the initial conditions $\Psi_{0} \in \mathcal{H}^{1}_{0}(\Omega)$, $\mathbf{A}_{0} \in \mathbf{H}^{1}_{t}(\Omega)$ and $\mathbf{A}_{t}(\cdot , 0) \in \mathbf{L}^{2}(\Omega)$.

Let M be a positive integer and let $\Delta t =T / M$  be the time step. For any k=1,2,$\cdots, M$, we introduce the following notation:
\begin{equation}\label{eq:2-2}
\begin{array}{@{}l@{}}
{\displaystyle \partial U^{k}= (U^{k}-U^{k-1})/\Delta t, \quad\partial^{2} U^{k}=(\partial U^{k}-\partial U^{k-1})/\Delta t,}\\[2mm]
{\displaystyle \overline{U}^{k} = (U^{k}+U^{k-1})/2, \quad \widetilde{U}^{k}=(U^{k}+U^{k-2})/2,}\\[2mm]
\end{array}
\end{equation}
for any given sequence $\{U^{k}\}_{0}^{M}$ and denote $u^{k}=u(\cdot,t^{k})$ for any given functions
$u\in C(0,T;\,X)$ with a Banach space $ X $.

Let $\mathcal{T}_{h}=\{K\}$ be a regular partition of $\Omega$ into triangles in $\mathbb{R}^{2}$ or tetrahedrons
in $\mathbb{R}^{3}$ without loss of generality, where the mesh size $h= \mathrm{max}_{K\in\mathcal{T}_{h}}\{diam(K)\}$. For any $K \in \mathcal{T}_{h}$, we denote by $P_{r}(K)$ the spaces of polynomials of degree $r \;(r \geq 1)$ defined on $K$. We now define the standard Lagrange finite element space
\begin{equation}
Y^{r}_{h} = \{u_{h} \in C(\overline{\Omega}):\; u_{h}|_{K}  \in P_{r}(K), \; \forall \; K \in  \mathcal{T}_{h}\}.
\end{equation}
We have the following finite element subspaces of $H_{0}^{1}(\Omega)$,  $\mathcal{H}_{0}^{1}(\Omega)$  and $\mathbf{H}^{1}_{\rm t}(\Omega)$ 
\begin{equation}
V^{r}_{h} = Y^{r}_{h} \cap  H_{0}^{1}(\Omega), \quad \mathcal{V}^{r}_{h} = V^{r}_{h} \oplus {\rm i}V^{r}_{h},\quad \mathbf{V}^{r}_{h} = \big( Y^{r}_{h}\big)^{3}\cap \mathbf{H}^{1}_{\rm t}(\Omega).
\end{equation}

We shall approximate the wave function $\Psi$ and the vector potential $\mathbf{A}$ by the functions in 
 $\mathcal{V}_{h}^{r}$ and $\mathbf{V}_{h}^{r}$, respectively. Let $I_h$ and $ {\bm \pi}_h  $ be the conventional pointwise interpolation operators on $\mathcal{V}_{h}^{r}$ and $\mathbf{V}_{h}^{r}$, respectively. For $0\leq s \leq m\leq r+1$, $m\geq 2$, $2\leq p \leq \infty$, standard finite element theory gives that \cite{Bre}:
 \begin{equation}\label{eq:2-3}
 \begin{array}{@{}l@{}}
{\displaystyle \Vert u - I_h u \Vert_{\mathcal{W}^{s}_{p}} \leq Ch^{m-s}\Vert u \Vert_{\mathcal{W}^{m}_{p}} \quad \forall \; u \in \mathcal{W}^{m}_{p}(\Omega) ,} \\[2mm]
{\displaystyle \Vert \mathbf{v} - {\bm \pi}_{h} \mathbf{v} \Vert_{\mathbf{W}^{s}_{p}} \leq Ch^{m-s}\Vert \mathbf{v} \Vert_{\mathbf{W}^{m}_{p}} \quad \forall \; \mathbf{v} \in \mathbf{W}^{m}_{p}(\Omega).}
\end{array}
\end{equation}

For convenience, assume that the function $\mathbf{A}$ is defined in the interval $[-\Delta t, T]$
in terms of the time variable $t$. We can compute $ \mathbf{A}(\cdot,-\Delta t) $ by
\begin{equation}\label{eq:2-4}
{\displaystyle \mathbf{A}(\cdot,-\Delta t)=\mathbf{A}(\cdot,0)-\Delta t \frac{\partial \mathbf{A}}{\partial t}(\cdot,0)=\mathbf{A}_{0}-\Delta t\mathbf{A}_{1}},
\end{equation}
which leads to an approximation to $\mathbf{A}^{-1}$ with second order accuracy.

A Crank-Nicolson Galerkin finite element approximation to the Maxwell-Schr\"{o}dinger system (\ref{eq:2-1}) is formulated as follows:
\begin{equation}\label{eq:2-5}
{\displaystyle \Psi_{h}^{0}=I_{h}\Psi_{0},\quad\mathbf{A}_{h}^{0}={\bm \pi}_h \mathbf{A}_{0},\quad
\mathbf{A}_{h}^{0}-\mathbf{A}_{h}^{-1}=\Delta t {\bm \pi}_h\mathbf{A}_{1},}
\end{equation}
and find $(\Psi_{h}^{k},\mathbf{A}_{h}^{k})\in\mathcal{V}_{h}^{r}\times\mathbf{V}^{r}_{h}$ such that for $k=1,2 ,\cdots, M$,
\begin{equation}\label{eq:2-6}
\left\{
\begin{array}{@{}l@{}}
{\displaystyle
 -\mathrm{i}(\partial\Psi_{h}^{k},\varphi)+
\frac{1}{2}\left((\mathrm{i}\nabla +\overline{\mathbf{A}}_{h}^{k})\overline{\Psi}_{h}^{k},(\mathrm{i}\nabla +\overline{\mathbf{A}}_{h}^{k})\varphi\right)
 + (V_{0}\overline{\Psi}_{h}^{k},\varphi) = 0 ,\quad \forall \varphi\in\mathcal{V}_{h}^{r}}\\[2mm]
{\displaystyle (\partial^{2}\mathbf{A}_{h}^{k},\mathbf{v}) +\left(\frac{\mathrm{i}}{2}\big((\Psi_{h}^{k-1})^{\ast}\nabla{\Psi_{h}^{k-1}}
-\Psi_h^{k-1}\nabla{(\Psi_{h}^{k-1})}^{\ast}\big),\mathbf{v}\right) + (\nabla\times
 \widetilde{\mathbf{A}}_{h}^{k},\nabla\times\mathbf{v})}\\[2mm]
{\displaystyle\quad +\gamma(\nabla\cdot
 \widetilde{\mathbf{A}}_{h}^{k},\nabla\cdot\mathbf{v})+\big(|\Psi_{h}^{k-1}|^{2}\frac{\overline{\mathbf{A}}_{h}^{k}
 +\overline{\mathbf{A}}_{h}^{k-1}}{2},\mathbf{v}\big)=( \mathbf{g}^{k-1}, \mathbf{v}), \quad \forall\mathbf{v}\in\mathbf{V}^{r}_{h},}
\end{array}
\right.
\end{equation}
where $ \overline{\mathbf{A}}_h^k $, $ \overline{\Psi}_h^k $ and $ \widetilde{\mathbf{A}}_h^k $ have been defined in (\ref{eq:2-2}).

For convenience, we define the following bilinear forms:
\begin{equation}\label{eq:2-7}
\begin{array}{@{}l@{}}
{\displaystyle B(\mathbf{A};\Psi,\varphi)=\left((\mathrm{i}\nabla+\mathbf{A})\Psi,(\mathrm{i}\nabla+\mathbf{A})\varphi\right),}\\[2mm]
{\displaystyle D(\mathbf{A},\mathbf{v})=\gamma(\nabla\cdot \mathbf{A},\nabla\cdot \mathbf{v})+
(\nabla\times\mathbf{A},\nabla\times\mathbf{v}),}\\[2mm]
{\displaystyle f(\Psi,\varphi)=\frac{\mathrm{i}}{2}(\varphi^{\ast}\nabla\Psi-\Psi\nabla\varphi^{\ast}).}
\end{array}
\end{equation}

Then the variational forms of the modified Maxwell-Schr\"{o}dinger equations and the discrete system can be written as follows:
\begin{equation}\label{eq:2-8}
\left\{
\begin{array}{@{}l@{}}
{\displaystyle -\mathrm{i}(\frac{\partial \Psi}{\partial t},\varphi)+\frac{1}{2}B(\mathbf{A};\Psi,\varphi)+(V_{0}\Psi,\varphi)=0,\quad \forall \varphi\in\mathcal{H}_{0}^{1}(\Omega),}\\[2mm]
{\displaystyle (\frac{\partial ^{2}\mathbf{A}}{\partial t^{2}},\mathbf{v})
+D(\mathbf{A},\mathbf{v})+( f(\Psi,\Psi),\mathbf{v})+( |\Psi|^{2}\mathbf{A},\mathbf{v})=( \mathbf{g}, \mathbf{v}), \quad \forall\mathbf{v}\in\mathbf{H}^{1}_{t}(\Omega),}
\end{array}
\right.
\end{equation}
and for $ k=1,2,\cdots,M$,
\begin{equation}\label{eq:2-9}
\left\{
\begin{array}{@{}l@{}}
{\displaystyle
 -\mathrm{i}(\partial\Psi_{h}^{k},\varphi)+
\frac{1}{2}B(\overline{\mathbf{A}}_{h}^{k};\overline{\Psi}_{h}^{k},\varphi)
 +( V_{0}\overline{\Psi}_{h}^{k},\varphi) = 0 ,\quad \forall \varphi\in\mathcal{V}_{h}^{r}}\\[2mm]
{\displaystyle (\partial ^{2}\mathbf{A}_{h}^{k},\mathbf{v})+D(\widetilde{\mathbf{A}}_{h}^{k},\mathbf{v})
+\left( f(\Psi_{h}^{k-1},\Psi_h^{k-1}),\mathbf{v}\right)}\\[2mm]
 {\displaystyle\quad+\big( |\Psi_{h}^{k-1}|^{2}\frac{(\overline{\mathbf{A}}_{h}^{k}+\overline{\mathbf{A}}_{h}^{k-1})}{2},\mathbf{v}\big)=( \mathbf{g}^{k-1}, \mathbf{v}), \quad
 \forall\mathbf{v}\in\mathbf{V}^{r}_{h}.}
\end{array}
\right.
\end{equation}

Note that after discretization in time and space, the Maxwell equation and Schr\"{o}dinger equation in the discrete system (\ref{eq:2-9}) are decoupled. At each time step, we only need to solve the two discrete linear equations alternately.     

In this paper we assume that the modified  Maxwell-Schr\"{o}dinger equations (\ref{eq:2-8}) has one and only one weak solution $(\Psi,\mathbf{A})$ and
the following regularity conditions are satisfied:
\begin{equation}\label{eq:2-10}
\begin{array}{@{}l@{}}
{\displaystyle
\Psi,\Psi_{t},\Psi_{tt} \in {L}^{\infty}(0, T; \mathcal{H}^{r+1}(\Omega)),\quad \Psi_{ttt} \in {L}^{\infty}(0, T; \mathcal{H}^{1}(\Omega)),}\\[2mm]
{\displaystyle\Psi_{tttt} \in L^{2}(0, T; \mathcal{L}^{2}(\Omega)),\quad
\mathbf{A},\mathbf{A}_{t},\mathbf{A}_{tt} \in {L}^{\infty}(0, T; \mathbf{H}^{r+1}(\Omega)),}\\[2mm]
{\displaystyle\mathbf{A}_{ttt} \in {L}^{\infty}(0, T; \mathbf{H}^{1}(\Omega)),\mathbf{A}_{tttt} \in L^{2}(0, T; \mathbf{L}^{2}(\Omega)).}
\end{array}
\end{equation}

For the initial conditions $(\Psi_{0},\mathbf{A}_{0},\mathbf{A}_{1})$ and the right hand function $\mathbf{g}(\mathbf{x},t)$, we assume that
\begin{equation}\label{eq:2-11}
\begin{array}{@{}l@{}}
{\displaystyle \Psi_{0}\in \mathcal{H}^{r+1}(\Omega) \cap \mathcal{H}_{0}^{1}(\Omega), \,\,\, \mathbf{A}_{0},\mathbf{A}_{1}\in\mathbf{H}^{r+1}(\Omega)\cap\mathbf{H}_{t}^{1}(\Omega),\,\,
\,\mathbf{g}\in C(0, T; \mathbf{L}^{2}(\Omega)).}
\end{array}
\end{equation}

We now give the main convergence result in this paper as follows:
\begin{theorem}\label{thm2-1}
Suppose that $ \Omega\subset \mathbb{R}^d  $, $ d\geq 2 $ is a bounded Lipschitz polyhedral convex domain.
Let $(\Psi,\mathbf{A})$ be the unique solution of the modified Maxwell-Schr\"{o}dinger equations (\ref{eq:2-8}), and let
$ (\Psi_h^k,\mathbf{A}_h^k)$ be the numerical solution of the full discrete scheme (\ref{eq:2-9}) associated with (\ref{eq:2-8}). Under the assumptions
(\ref{eq:2-10}) and (\ref{eq:2-11}), we have the following error estimates
\begin{equation}\label{eq:2-12}
\begin{array}{@{}l@{}}
 {\displaystyle \max_{1\leq k \leq M}\big\{\|\Psi_{h}^{k}-\Psi^{k}\|_{\mathcal{L}^2(\Omega)}^{2}
 +\|\nabla(\Psi_{h}^{k}-\Psi^{k})\|_{\mathbf{L}^2(\Omega)}^{2}
 + \|\mathbf{A}_{h}^{k}-\mathbf{A}^{k}\|_{\mathbf{L}^2(\Omega)}^{2}}\\[2mm]
{\displaystyle \quad+\|\nabla\cdot(\mathbf{A}_{h}^{k}-\mathbf{A}^{k})\|_{{L}^2(\Omega)}^{2}
+\|\nabla\times(\mathbf{A}_{h}^{k}-\mathbf{A}^{k})\|_{\mathbf{L}^2(\Omega)}^{2}\big\}
\leq C \big\{h^{2r}+(\Delta t)^{4}\big\},}
\end{array}
\end{equation}
where $ \Psi^k=\Psi(\cdot, t^k) $, $ \mathbf{A}^k=\mathbf{A}(\cdot, t^k) $,
and $C$ is a constant independent of $h$, $\Delta t$.
\end{theorem}

The proof of Theorem~\ref{thm2-1} will be given in section~\ref{sec-4}.

\section{ Stability estimates}\label{sec-3}

In this section we derive some stability estimates for the numerical solutions of the full discrete system (\ref{eq:2-9}), which play an important role in the error estimates in the next section.

For convenience, we list the following imbedding inequalities and interpolation inequalities in Sobolev spaces (see, e.g., \cite{Lad} and \cite{Gir}),
and use them in the sequel:
\begin{equation}\label{eq:3-1}
 {\displaystyle \|u\|_{L^p} \leq C \|u\|_{{H}^1}, \quad \|\mathbf{v}\|_{\mathbf{L}^p}\leq C \|\mathbf{v}\|_{\mathbf{H}^1},
 \quad 1\leq p \leq 6 \,\,(d=2,3),}
 \end{equation}
 \begin{equation}\label{eq:3-2}
 {\displaystyle  \|\mathbf{v}\|_{\mathbf{H}^1} \leq C(\|\nabla\times\mathbf{v}\|_{\mathbf{L}^2}+\|\nabla\cdot\mathbf{v}\|_{{L}^2}),\quad
 \mathbf{v}\in \mathbf{H}^{1}_{t}(\Omega),}
 \end{equation}
 \begin{equation}\label{eq:3-3}
 {\displaystyle \|u\|_{{L}^3}^{2}\leq \|u\|_{{L}^2}\|u\|_{{L}^6},}
 \end{equation}
where $ \|u\|_{L^p}=\|u\|_{L^p(\Omega)} $, $ \|u\|_{{H}^1}=\|u\|_{{H}^1(\Omega)} $,
$ \|\mathbf{v}\|_{\mathbf{L}^2}=\|\mathbf{v}\|_{\mathbf{L}^2(\Omega)} $ and
$\|\mathbf{v}\|_{\mathbf{H}^1}=\|\mathbf{v}\|_{\mathbf{H}^1(\Omega)} $.

We first give the definition of the discrete energy functional of (\ref{eq:2-9}) as follows:
\begin{equation}\label{eq:3-8}
\begin{array}{@{}l@{}}
{\displaystyle \mathcal{G}_{h}^{k}=\frac{1}{2}B(\overline{\mathbf{A}}_{h}^{k};\Psi_{h}^{k},\Psi_{h}^{k})+\frac{1}{4}D(\mathbf{A}_{h}^{k},\mathbf{A}_{h}^{k})
+\frac{1}{4}D(\mathbf{A}_{h}^{k-1},\mathbf{A}_{h}^{k-1})+V_0\|\Psi_{h}^{k}\|_{\mathcal{L}^2}^{2} }\\[2mm]
{\displaystyle \quad+\frac{1}{2} \|\partial \mathbf{A}_{h}^{k}\|_{\mathbf{L}^2}^{2}=\frac{1}{2}B(\overline{\mathbf{A}}_{h}^{k};\Psi_{h}^{k},\Psi_{h}^{k})
+\frac{1}{2}\overline{D(\mathbf{A}_{h}^{k},\mathbf{A}_{h}^{k})}+V_0\|\Psi_{h}^{k}\|_{\mathcal{L}^2}^{2} + \frac{1}{2}\|\partial \mathbf{A}_{h}^{k}\|_{\mathbf{L}^2}^{2},}
\end{array}
\end{equation}
where $ \overline{D(\mathbf{A}_{h}^{k},\mathbf{A}_{h}^{k})}=\frac{\displaystyle 1}{\displaystyle 2}\Big[
D(\mathbf{A}_{h}^{k},\mathbf{A}_{h}^{k})+D(\mathbf{A}_{h}^{k-1},\mathbf{A}_{h}^{k-1})\Big] $.

\begin{lemma}\label{lem3-2}
For the solution of (\ref{eq:2-9}), for $k=1,2\cdots,M $, we have
\begin{equation}\label{eq:3-9}
{\displaystyle {\|\Psi_{h}^{k}\|}_{\mathcal{L}^2}^{2}={\|\Psi_{h}^{0}\|}_{\mathcal{L}^2}^{2},\quad \mathcal{G}_{h}^{k}\leq C,}
\end{equation}
where C is a constant independent of k, $h$ and $\Delta t $.
\end{lemma}

\begin{proof}
Choosing $\varphi=\overline{\Psi}_{h}^{k}$ in $(\ref{eq:2-9})_1 $ and taking the imaginary part, we can complete the proof of $(\ref{eq:3-9})_1 $.
Let us turn to the proof of $(\ref{eq:3-9})_2 $. It is obvious that
\begin{equation*}
\begin{array}{@{}l@{}}
{\displaystyle \mathrm{Re}\left[B\left(\overline{\mathbf{A}}_{h}^{k};\overline{\Psi}_{h}^{k},\partial \Psi_{h}^{k}\right)\right]
=\frac{1}{2}\partial B(\overline{\mathbf{A}}_{h}^{k};{\Psi}_{h}^{k},\Psi_{h}^{k}) }\\[2mm]
{\displaystyle \quad\quad+ \frac{1}{2\Delta t}\left[B(\overline{\mathbf{A}}_{h}^{k-1};{\Psi}_{h}^{k-1},\Psi_{h}^{k-1})
-B(\overline{\mathbf{A}}_{h}^{k};{\Psi}_{h}^{k-1},\Psi_{h}^{k-1})\right]}\\[2mm]
{\displaystyle \quad\quad+ \frac{1}{2 \Delta t}\mathrm{Re}\left[B(\overline{\mathbf{A}}_{h}^{k};\Psi_{h}^{k-1},\Psi_{h}^{k})
-B(\overline{\mathbf{A}}_{h}^{k};\Psi_{h}^{k},\Psi_{h}^{k-1})\right].}
\end{array}
\end{equation*}

By a direct computation, we get
\begin{equation}\label{eq:3-10}
\begin{array}{@{}l@{}}
{\displaystyle B(\mathbf{A};\psi,\varphi)=(\nabla\psi,\nabla\varphi)+(\mathbf{A}\psi,\mathbf{A}\varphi)+2(f(\psi,\varphi),\mathbf{A}),}\\[2mm]
{\displaystyle B(\mathbf{A};\psi,\varphi)-B(\hat{\mathbf{A}};\psi,\varphi)=\left((\mathbf{A}
+\hat{\mathbf{A}})\psi \varphi^{*},\mathbf{A}-\hat{\mathbf{A}}\right)
+2(f(\psi,\varphi),\mathbf{A}-\hat{\mathbf{A}}),}
\end{array}
\end{equation}
and consequently
\begin{equation*}
{\displaystyle\mathrm{Re}\left[B(\overline{\mathbf{A}}_{h}^{k};\Psi_{h}^{k-1},\Psi_{h}^{k})-B(\overline{\mathbf{A}}_{h}^{k};\Psi_{h}^{k},\Psi_{h}^{k-1})\right] = 0.}
\end{equation*}

We thus have
\begin{equation}\label{eq:3-11}
\begin{array}{@{}l@{}}
{\displaystyle \mathrm{Re}\left[B\left(\overline{\mathbf{A}}_{h}^{k};\overline{\Psi}_{h}^{k},\partial \Psi_{h}^{k}\right)\right]
=-\left(|\Psi_{h}^{k-1}|^{2}\frac{\overline{\mathbf{A}}_{h}^{k}+\overline{\mathbf{A}}_{h}^{k-1}}{2},\frac{\overline{\mathbf{A}}_{h}^{k}
-\overline{\mathbf{A}}_{h}^{k-1}}{\Delta t}\right)
}\\[2mm]
{\displaystyle \quad+\frac{1}{2}\partial B(\overline{\mathbf{A}}_{h}^{k};{\Psi}_{h}^{k},\Psi_{h}^{k})-\left(f(\Psi_{h}^{k-1},\Psi_{h}^{k-1}),\frac{\overline{\mathbf{A}}_{h}^{k}
-\overline{\mathbf{A}}_{h}^{k-1}}{\Delta t}\right).}
\end{array}
\end{equation}

It is not difficult to check that
\begin{equation}\label{eq:3-12}
{\displaystyle \mathrm{Re}\left[V_0(\overline{\Psi}_{h}^{k},\partial \Psi_{h}^{k})\right]=\frac{V_{0}}{2}\partial(\Psi_{h}^{k},\Psi_{h}^{k}).}
\end{equation}

We choose $\varphi=\partial\Psi_{h}^{k}$ in $ (\ref{eq:2-9})_1 $ and take the real part. Combining (\ref{eq:3-11}) and (\ref{eq:3-12}) gives
\begin{equation}\label{eq:3-13}
\begin{array}{@{}l@{}}
{\displaystyle \frac{1}{2}\partial\|\left(\mathrm{i}\nabla+\overline{\mathbf{A}}_{h}^{k}\right)\Psi_{h}^{k}\|_{\mathbf{L}^2}^{2}+V_{0}\partial \|\Psi_{h}^{k}\|_{\mathcal{L}^2}^{2} -\left(|\Psi_{h}^{k-1}|^{2}\frac{\overline{\mathbf{A}}_{h}^{k}+\overline{\mathbf{A}}_{h}^{k-1}}{2},\frac{\overline{\mathbf{A}}_{h}^{k}
-\overline{\mathbf{A}}_{h}^{k-1}}{\Delta t}\right) }\\[2mm]
{\displaystyle \quad -\left(f(\Psi_{h}^{k-1},\Psi_{h}^{k-1}),\frac{\overline{\mathbf{A}}_{h}^{k}-\overline{\mathbf{A}}_{h}^{k-1}}{\Delta t}\right)=0.}
\end{array}
\end{equation}

Taking $\mathbf{v}=\frac{\displaystyle 1}{\displaystyle 2\Delta t}(\mathbf{A}_{h}^{k}-\mathbf{A}_{h}^{k-2})=\frac{\displaystyle 1}
{\displaystyle \Delta t}
(\overline{\mathbf{A}}_{h}^{k}-\overline{\mathbf{A}}_{h}^{k-1})=\frac{\displaystyle 1}{\displaystyle 2}(\partial \mathbf{A}_{h}^{k}+\partial \mathbf{A}_{h}^{k-1})$
in $ (\ref{eq:2-9})_2 $, and combining with (\ref{eq:3-13}), we get
\begin{equation*}
\begin{array}{@{}l@{}}
{\displaystyle \partial \left(\frac{1}{2} B(\overline{\mathbf{A}}_{h}^{k};\Psi_{h}^{k},\Psi_{h}^{k})+V_0\|\Psi_{h}^{k}\|_{\mathcal{L}^2}^{2}
 +\frac{1}{2}\|\partial \mathbf{A}_{h}^{k}\|_{\mathbf{L}^2}^{2}\right)}\\[2mm]
{\displaystyle \quad+\partial\left(\frac{1}{4}D(\mathbf{A}_{h}^{k},\mathbf{A}_{h}^{k})
+\frac{1}{4}D(\mathbf{A}_{h}^{k-1},\mathbf{A}_{h}^{k-1})\right)=(\mathbf{g}^{k-1},\frac{1}{2}(\partial \mathbf{A}_{h}^{k}+\partial \mathbf{A}_{h}^{k-1})) .}
\end{array}
\end{equation*}
It follows that
\begin{equation}\label{eq:3-13-0}
\begin{array}{@{}l@{}}
{\displaystyle \partial \mathcal{G}_{h}^{k} \leq \Vert \mathbf{g}^{k-1} \Vert_{\mathbf{L}^{2}} \Vert
\frac{1}{2}(\partial \mathbf{A}_{h}^{k}+\partial \mathbf{A}_{h}^{k-1})\Vert_{\mathbf{L}^{2}} }\\[2mm]
{\displaystyle \qquad \leq C\left(\Vert \mathbf{g}^{k-1} \Vert_{\mathbf{L}^{2}}^{2} + \Vert
\partial \mathbf{A}_{h}^{k}\Vert_{\mathbf{L}^{2}}^{2} + \Vert
\partial \mathbf{A}_{h}^{k-1}\Vert_{\mathbf{L}^{2}}^{2}\right).}
\end{array}
\end{equation}
 Multiply (\ref{eq:3-13-0}) by $\Delta t$, sum $k=1,2,\cdots,M$, to discover
\begin{equation*}
\begin{array}{@{}l@{}}
{\displaystyle \mathcal{G}_{h}^{M} \leq \mathcal{G}_{h}^{0} + C\Delta t \sum_{k=0}^{M-1} \Vert \mathbf{g}^{k-1} \Vert_{\mathbf{L}^{2}}^{2} + C \Delta t \sum_{k=0}^{M}\mathcal{G}_{h}^{k}}.
\end{array}
\end{equation*}

Now  $(\ref{eq:3-9})_2 $ follows from the discrete Gronwall's inequality and  thus we complete the proof of Lemma~\ref{lem3-2}. \qquad\end{proof}

\begin{rem}\label{rem3-2}
 Lemma~\ref{lem3-2} shows that the numerical scheme presented in this paper for  the modified Maxwell-Schr\"{o}dinger equations (\ref{eq:2-8}) is stable in some senses.
\end{rem}

\begin{theorem}\label{thm3-1}
The solution of the full discrete system (\ref{eq:2-9}) fulfills the following estimates
\begin{equation}\label{eq:3-14}
\begin{array}{@{}l@{}}
{\displaystyle\|\left(\mathrm{i}\nabla+\overline{\mathbf{A}}_{h}^{k}\right)\Psi_{h}^{k}\|_{\mathbf{L}^2}
+\|\Psi_{h}^{k}\|_{\mathcal{L}^2}+\|\mathbf{A}_{h}^{k}\|_{\mathbf{L}^2}+\|\partial \mathbf{A}_{h}^{k}\|_{\mathbf{L}^2}}\\[2mm]
{\displaystyle\quad +\|\nabla \times\mathbf{A}_{h}^{k}\|_{\mathbf{L}^2}
+\gamma\|\nabla \cdot\mathbf{A}_{h}^{k}\|_{{L}^2}\leq C,}
\end{array}
\end{equation}
 and
\begin{equation}\label{eq:3-15}
{\displaystyle\|\mathbf{A}_{h}^{k}\|_{\mathbf{H}^1}
+\|\mathbf{A}_{h}^{k}\|_{\mathbf{L}^6}
+\|\Psi_{h}^{k}\|_{\mathcal{H}^1}+\|\Psi_{h}^{k}\|_{\mathcal{L}^6}\leq C,}
\end{equation}
where $ C $ is a constant independent of $ h $, $ \Delta t $.
\end{theorem}

\begin{proof}
(\ref{eq:3-14}) is the direct result of Lemma ~\ref{lem3-2}. Next we give the proof of (\ref{eq:3-15}).
Since the semi-norm in $\mathbf{H}_t^{1}(\Omega)$ is equivalent to $\mathbf{H}^{1}$-norm, from (\ref{eq:3-14}) we get
\begin{equation}\label{eq:3-16}
\Vert \mathbf{A}^{k}_h\|_{\mathbf{H}^1} < C.
\end{equation}
Then Sobolev's imbedding theorem implies that
\begin{equation}\label{eq:3-17}
\Vert \mathbf{A}^{k}_h\|_{\mathbf{L}^p} < C ,
\end{equation}
with $ \, 1\leq p \leq 6$ for $d = 3$ and $ 1\leq p < \infty $ for $ d=2$.

Using  Young's inequality and the interpolation inequalities (\ref{eq:3-3}), we further prove
\begin{equation*}
\begin{array}{@{}l@{}}
{\displaystyle \|{\mathbf{A}}^{k}_{h}\Psi^{k}_{h}\|_{\mathbf{L}^2}\leq \|{\mathbf{A}}^{k}_{h}\|_{\mathbf{L}^6} \|\Psi^{k}_{h}\|_{\mathcal{L}^3} \leq C \|\Psi^{k}_{h}\|_{\mathcal{L}^3}
}\\[2mm]
{\displaystyle \quad\leq C\|\Psi^{k}_{h}\|_{\mathcal{L}^2}^{\frac{1}{2}}\|\Psi^{k}_{h}\|_{\mathcal{L}^6}^{\frac{1}{2}} \leq C\|\Psi^{k}_{h}\|_{\mathcal{L}^2}^{\frac{1}{2}}\|\nabla\Psi^{k}_{h}\|_{\mathbf{L}^2}^{\frac{1}{2}}
\leq C+\frac{1}{2}\|\nabla\Psi^{k}_{h}\|_{\mathbf{L}^2}}.
\end{array}
\end{equation*}

Hence we have
\begin{equation*}
\begin{array}{@{}l@{}}
{\displaystyle \|\nabla\Psi^{k}_{h}\|_{\mathbf{L}^2}\leq \|\left(\mathrm{i}\nabla
+{\mathbf{A}}^{k}_{h}\right)\Psi^{k}_{h}\|_{\mathbf{L}^2}
+ \|{\mathbf{A}}^{k}_{h}\Psi^{k}_{h}\|_{\mathbf{L}^2}
\leq C+\frac{1}{2}\|\nabla\Psi^{k}_{h}\|_{\mathbf{L}^2}.}
\end{array}
\end{equation*}
Consequently, we obtain
\begin{equation}\label{eq:3-18}
\|\Psi^{k}_{h}\|_{\mathcal{H}_{0}^1} + \|\Psi^{k}_{h}\|_{\mathcal{L}^6} \leq C.
\end{equation}
Combining (\ref{eq:3-16}), (\ref{eq:3-17}) and (\ref{eq:3-18}), we complete the proof of (\ref{eq:3-15}). $\qquad$\end{proof}

\section{The error estimates}\label{sec-4}

In this section, we will give the proof of Theorem~\ref{thm2-1}.
Let $(I_{h}\Psi, {\bm \pi}_{h}\mathbf{A})$ denote the interpolation functions of $(\Psi, \mathbf{A})$
in $\mathcal{V}_{h}^{r}\times\mathbf{V}_{h}^{r}$. Set
$e_{\Psi}=I_{h}\Psi-\Psi$,  $e_{\mathbf{A}}={\bm \pi}_h\mathbf{A}-\mathbf{A}$.
By applying the interpolation error estimates (\ref{eq:2-3}) and the regularity assumptions (\ref{eq:2-10}), we have
\begin{equation}\label{eq:4-1}
\begin{array}{@{}l@{}}
{\displaystyle \|e_{\Psi}\|_{\mathcal{L}^2}+h\|e_{\Psi}\|_{\mathcal{H}^1}\leq Ch^{r+1},\quad  \|e_{\mathbf{A}}\|_{\mathbf{L}^2}
+h\|e_{\mathbf{A}}\|_{\mathbf{H}^1}\leq C h^{r+1},}\\[2mm]
{\displaystyle
\|I_{h}\Psi\|_{\mathcal{L}^{\infty}}+\|{\bm \pi}_h\mathbf{A}\|_{\mathbf{H}^{1}}
 +\|\nabla I_{h}\Psi \|_{\mathbf{L}^{3}}\leq C,}
\end{array}
\end{equation}
where $ C $ is a constant independent of $ h $.

For convenience, we give the following identities, which will be used frequently in the sequel.
\begin{equation}\label{eq:4-2}
\begin{array}{@{}l@{}}
{\displaystyle \sum_{k=1}^{M}{(a_{k}-a_{k-1})b_{k}}=a_{M}b_{M}-a_{0}b_{1}-\sum_{k=1}^{M-1}{a_{k}(b_{k+1}-b_{k})},}\\[2mm]
{\displaystyle \sum_{k=1}^{M}{(a_{k}-a_{k-1})b_{k}}=a_{M}b_{M}-a_{0}b_{0}-\sum_{k=1}^{M}{a_{k-1}(b_{k}-b_{k-1})}.}
\end{array}
\end{equation}

Let $\theta_{\Psi}^{k}=\Psi_{h}^{k}-I_{h}\Psi^{k} $, $\theta_{\mathbf{A}}^{k}=\mathbf{A}_{h}^{k}-{\bm \pi}_{h}\mathbf{A}^{k}$.
By using the error estimates of the interpolation operators, we only need to estimate $\theta_{\Psi}^{k}$ and $\theta_{\mathbf{A}}^{k}$.
Subtracting (\ref{eq:2-8}) from (\ref{eq:2-9}), we get the following equations for $\theta_{\Psi}^{k}$ and $\theta_{\mathbf{A}}^{k}$:
\begin{equation}\label{eq:4-3}
\begin{array}{@{}l@{}}
{\displaystyle -2\mathrm{i}(\partial \theta^{k}_{\Psi},\varphi)+B\left(\overline{\mathbf{A}}^{k}_{h};\overline{\theta}_{\Psi}^{k},\varphi\right)
= 2\mathrm{i}\left(\partial I_{h}\Psi^{k} - (\Psi_{t})^{k-\frac{1}{2}},\varphi\right)}\\[2mm]
{\displaystyle \quad+2V_0\left(\Psi^{k-\frac{1}{2}}-\overline{\Psi}_{h}^{k},\varphi\right)+ B(\mathbf{A}^{k-\frac{1}{2}};(\Psi^{k-\frac{1}{2}}-I_{h}\overline{\Psi}^{k}),\varphi)}\\[2mm]
{\displaystyle \quad+\left(B(\mathbf{A}^{k-\frac{1}{2}};I_{h}\overline{\Psi}^{k},\varphi)-B(\overline{\mathbf{A}}^{k}_{h}; I_{h}\overline{\Psi}^{k},\varphi)\right),\quad \forall \varphi\in\mathcal{V}_{h}^{r},}
\end{array}
\end{equation}
and
\begin{equation}\label{eq:4-4}
\begin{array}{@{}l@{}}
{\displaystyle \left(\partial^{2}\theta^{k}_{\mathbf{A}},\mathbf{v}\right)+D(\widetilde{\theta^{k}_{\mathbf{A}}},\mathbf{v})
= \left((\mathbf{A}_{tt})^{k-1}-\partial^{2} {\bm \pi}_{h}\mathbf{A}^{k},\mathbf{v}\right)}\\[2mm]
{\displaystyle \quad +D(\mathbf{A}^{k-1}-\widetilde{{\bm \pi}_{h}\mathbf{A}^{k}},\mathbf{v})
+ \left(|\Psi^{k-1}|^{2}\mathbf{A}^{k-1}-|\Psi^{k-1}_{h}|^{2}\frac{\overline{\mathbf{A}}^{k}_{h}+\overline{\mathbf{A}}^{k-1}_{h}}{2},
\;\mathbf{v}\right)}\\[2mm]
{\displaystyle \quad+ \left(f(\Psi^{k-1},\Psi^{k-1})-f(\Psi^{k-1}_{h},\Psi^{k-1}_{h}),\;\mathbf{v}\right),\quad \forall\mathbf{v}\in\mathbf{V}^{r}_{h},}
\end{array}
\end{equation}
where $ \overline{\mathbf{A}}_h^k $, $ \overline{\Psi}_h^k $, $ \widetilde{\theta^{k}_{\mathbf{A}}}$
and $ \widetilde{{\bm \pi}_{h}\mathbf{A}^{k}} $ are similarly given in (\ref{eq:2-2}).

Now we briefly describe the outline of the proof of (\ref{eq:2-12}). First, we take $\varphi=\overline{\theta}_{\Psi}^{k}$ in (\ref{eq:4-3}) and
 obtain the estimate of $ \|\theta_\Psi^M\|_{\mathcal{L}^2} $. Second, we choose $\varphi=\theta_{\Psi}^{k}-\theta_{\Psi}^{k-1}$ in (\ref{eq:4-3})
 and derive the energy-norm estimate for $\theta^{M}_{\Psi}$. Finally, let $\mathbf{v}
  =\frac{\displaystyle 1}{\displaystyle 2\Delta t}(\theta_{\mathbf{A}}^{k}-\theta_{\mathbf{A}}^{k-2})$ in (\ref{eq:4-4})
  and acquire the estimate involving $\theta_{\mathbf{A}}^k $.  Combining the above three estimates, we will complete the proof of (\ref{eq:2-12}).

\subsection{Estimates for (\ref{eq:4-3})}

To begin with, choosing $\varphi=\overline{\theta}_{\Psi}^{k}$, $ k=1,2,\cdots, M $ as the test function in (\ref{eq:4-3}), we get
 \begin{equation}\label{eq:4-5}
 {\displaystyle 2\mathrm{i}(\partial \theta^{k}_{\Psi},\overline{\theta}_{\Psi}^{k})-B\left(\overline{\mathbf{A}}^{k}_{h};\overline{\theta}_{\Psi}^{k},\overline{\theta}_{\Psi}^{k}\right)
 =I_1^{(k)}+I_2^{(k)}+I_3^{(k)}+I_4^{(k)},}
 \end{equation}
 where
 \begin{equation*}
 \begin{array}{@{}l@{}}
 {\displaystyle I_1^{(k)}=2\mathrm{i}\left((\Psi_t)^{k-\frac{1}{2}}-\partial I_{h}\Psi^{k},\overline{\theta}_{\Psi}^{k}\right),\quad I_2^{(k)}=2V_0\left(\overline{\Psi}_{h}^{k}-\Psi^{k-\frac{1}{2}},\overline{\theta}_{\Psi}^{k}\right),}\\[2mm]
 {\displaystyle I_3^{(k)}=B(\mathbf{A}^{k-\frac{1}{2}};(I_{h}\overline{\Psi}^{k}-\Psi^{k-\frac{1}{2}}),\overline{\theta}_{\Psi}^{k}),
 \quad I_4^{(k)}=B(\overline{\mathbf{A}}^{k}_{h};I_{h}\overline{\Psi}^{k},\overline{\theta}_{\Psi}^{k})
 -B(\mathbf{A}^{k-\frac{1}{2}};I_{h}\overline{\Psi}^{k},\overline{\theta}_{\Psi}^{k}).}
 \end{array}
 \end{equation*}

Using the error estimates (\ref{eq:4-1}) for the interpolation operator $I_{h}$ and
the regularity of $\Psi$ in (\ref{eq:2-10}), it is easy to see that
 \begin{equation}\label{eq:4-6}
 \begin{array}{@{}l@{}}
{\displaystyle  |I_1^{(k)}|+ |I_2^{(k)}|\leq C\left((\Delta t)^{4}+h^{2r+2}\right)+C\left(\|\theta_{\Psi}^{k}\|^{2}_{\mathcal{L}^2}+\|\theta_{\Psi}^{k-1}\|^{2}_{\mathcal{L}^2}\right).}
\end{array}
\end{equation}

We observe that
\begin{equation}\label{eq:4-7}
\begin{array}{@{}l@{}}
 {\displaystyle B(\mathbf{A};\psi,\varphi)=(\nabla\psi,\nabla\varphi)+\left(|\mathbf{A}|^{2}\psi,\varphi\right)
 +i\left(\varphi^{\ast}\nabla\psi-\psi\nabla\varphi^{\ast},\mathbf{A}\right)}\\[2mm]
 {\displaystyle \quad \leq \|\nabla\psi\|_{\mathbf{L}^2}\|\nabla\varphi\|_{\mathbf{L}^2}+\|\mathbf{A}\|^{2}_{\mathbf{L}^6}\|\psi\|_{\mathcal{L}^6}\|\varphi\|_{\mathcal{L}^2}
 +\|\mathbf{A}\|_{\mathbf{L}^6}\big(\|\psi\|_{\mathcal{L}^3}\|\nabla\varphi\|_{\mathbf{L}^2}}\\[2mm]
 {\displaystyle \quad+\|\nabla\psi\|_{\mathbf{L}^2}\|\varphi\|_{\mathcal{L}^3}\big)\leq C\|\nabla\psi\|_{\mathbf{L}^2}
 \|\nabla\varphi\|_{\mathbf{L}^2},\quad \forall \mathbf{A}\in \mathbf{L}^6(\Omega),
 \,\,\,\psi,\varphi\in\mathcal{H}_{0}^{1}(\Omega),}
 \end{array}
 \end{equation}
 and
 \begin{equation}\label{eq:4-8}
 \begin{array}{@{}l@{}}
 {\displaystyle I_3^{(k)}= B(\mathbf{A}^{k-\frac{1}{2}};(I_{h}\overline{\Psi}^{k}-\overline{\Psi}^{k}),\overline{\theta}_{\Psi}^{k})
 +B(\mathbf{A}^{k-\frac{1}{2}};(\overline{\Psi}^{k}-\Psi^{k-\frac{1}{2}}),\overline{\theta}_{\Psi}^{k}).}
 \end{array}
 \end{equation}

 It follows from (\ref{eq:4-1}), (\ref{eq:4-7}) and (\ref{eq:4-8}) that
 \begin{equation}\label{eq:4-9}
 {\displaystyle |I_3^{(k)}|\leq C\left(h^{2r}+(\Delta t)^{4}\right)+C\left(\|\nabla\theta_{\Psi}^{k}\|^{2}_{\mathbf{L}^2}
 +\|\nabla\theta_{\Psi}^{k-1}\|^{2}_{\mathbf{L}^2}\right).}
 \end{equation}

Notice that
 \begin{equation}\label{eq:4-10}
 \begin{array}{@{}l@{}}
 {\displaystyle I_4^{(k)}= \Big[B(\overline{\mathbf{A}}^{k}_{h}; I_{h}\overline{\Psi}^{k},\overline{\theta}_{\Psi}^{k})
 -B({\bm\pi}_{h}\overline{\mathbf{A}}^{k}; I_{h}\overline{\Psi}^{k},\overline{\theta}_{\Psi}^{k})\Big]+\Big[B({\bm \pi}_{h}\overline{\mathbf{A}}^{k}; I_{h}\overline{\Psi}^{k},\overline{\theta}_{\Psi}^{k})}\\[2mm]
 {\displaystyle \quad
 -B(\overline{\mathbf{A}}^{k}; I_{h}\overline{\Psi}^{k},\overline{\theta}_{\Psi}^{k})\Big]+\Big[B(\overline{\mathbf{A}}^{k}; I_{h}\overline{\Psi}^{k},\overline{\theta}_{\Psi}^{k})
 -B(\mathbf{A}^{k-\frac{1}{2}}; I_{h}\overline{\Psi}^{k},\overline{\theta}_{\Psi}^{k})\Big]}\\[2mm]
 \end{array}
 \end{equation}

 By using (\ref{eq:3-1})-(\ref{eq:3-3}) and (\ref{eq:3-10}), we prove
\begin{equation}\label{eq:4-13}
 \begin{array}{@{}l@{}}
{\displaystyle |I_4^{(k)}|\leq C\left(D({\theta}^{k}_{\mathbf{A}},{\theta}^{k}_{\mathbf{A}})
+D({\theta}^{k-1}_{\mathbf{A}},{\theta}^{k-1}_{\mathbf{A}})\right)}\\[2mm]
{\displaystyle\qquad\quad+C\big\{h^{2r}+(\Delta t)^{4}+\|\nabla\theta_{\Psi}^{k}\|_{\mathbf{L}^2}^{2}+\|\nabla\theta_{\Psi}^{k-1}\|_{\mathbf{L}^2}^{2}\big\}.}
\end{array}
\end{equation}

Taking the imaginary part of (\ref{eq:4-5}), we have
\begin{equation}\label{eq:4-14}
\begin{array}{@{}l@{}}
{\displaystyle \frac{1}{\Delta t}\big(\|\theta^{k}_{\Psi}\|^{2}_{\mathcal{L}^2}
-\|\theta^{k-1}_{\Psi}\|^{2}_{\mathcal{L}^2}\big)
=\mathrm{Im}(I_1^{(k)})+\mathrm{Im}(I_2^{(k)})+\mathrm{Im}(I_3^{(k)})+\mathrm{Im}(I_4^{(k)})}\\[2mm]
{\displaystyle\qquad\quad\leq |I_1^{(k)}|+|I_2^{(k)}|+|I_3^{(k)}|+|I_4^{(k)}|
\leq C \left(D ({\theta}^{k}_{\mathbf{A}},{\theta}^{k}_{\mathbf{A}})
+D({\theta}^{k-1}_{\mathbf{A}},{\theta}^{k-1}_{\mathbf{A}})\right)}\\[2mm]
{\displaystyle\qquad\quad+C\big\{h^{2r}+(\Delta t)^{4}
+\|\nabla\theta_{\Psi}^{k}\|_{\mathbf{L}^2}^{2}
+\|\nabla\theta_{\Psi}^{k-1}\|_{\mathbf{L}^2}^{2}\big\},}
\end{array}
\end{equation}
and therefore
\begin{equation}\label{eq:4-15}
\begin{array}{@{}l@{}}
 {\displaystyle \|\theta^{M}_{\Psi}\|^{2}_{\mathcal{L}^2}\leq \|\theta^{0}_{\Psi}\|^{2}_{\mathcal{L}^2}
 +C\big(h^{2r}+(\Delta t)^{4}\big)+C \Delta t
 \sum_{k=1}^{M}{D({\theta}^{k}_{\mathbf{A}},{\theta}^{k}_{\mathbf{A}})}
 +C\Delta t\sum_{k=1}^{M} {\|\nabla\theta_{\Psi}^{k}\|_{\mathbf{L}^2}^{2}}}\\[2mm]
 {\displaystyle \quad \leq C\big(h^{2r}+(\Delta t)^{4}\big)
 +C\Delta t\sum_{k=1}^{M}{D({\theta}^{k}_{\mathbf{A}},
 {\theta}^{k}_{\mathbf{A}})}+C\Delta t\sum_{k=1}^{M} {\|\nabla\theta_{\Psi}^{k}\|_{\mathbf{L}^2}^{2}}.}
 \end{array}
 \end{equation}
 Here we have used the fact
 $\|\theta^{0}_{\Psi}\|^{2}_{\mathcal{L}^2} \leq C h^{2r+2} $.

To proceed further, we take $\varphi=\Delta t\partial{\theta_{\Psi}^{k}}=\theta_{\Psi}^{k}-\theta_{\Psi}^{k-1}$, $ k=1,2,\cdots, M $ in (\ref{eq:4-3}),  to find
\begin{equation}\label{eq:4-16}
 {\displaystyle -2\mathrm{i}\Delta t(\partial \theta^{k}_{\Psi},\partial{\theta_{\Psi}^{k}})+\Delta t B\left(\overline{\mathbf{A}}^{k}_{h};\overline{\theta}_{\Psi}^{k},\partial{\theta_{\Psi}^{k}}\right)
 =J_1^{(k)}+J_2^{(k)}+J_3^{(k)}+J_4^{(k)},}
 \end{equation}
 where
 \begin{equation*}
 \begin{array}{@{}l@{}}
 {\displaystyle J_1^{(k)}=2\mathrm{i}\left(\partial I_{h}\Psi^{k}-(\Psi_{t})^{k-\frac{1}{2}},\theta_{\Psi}^{k}-\theta_{\Psi}^{k-1}\right),\quad J_2^{(k)}=2V_0\left(\Psi^{k-\frac{1}{2}}-\overline{\Psi}_{h}^{k},\theta_{\Psi}^{k}-\theta_{\Psi}^{k-1}\right),}\\[2mm]
 {\displaystyle J_3^{(k)}=B(\mathbf{A}^{k-\frac{1}{2}};(\Psi^{k-\frac{1}{2}}-I_{h}\overline{\Psi}^{k}),\theta_{\Psi}^{k}-\theta_{\Psi}^{k-1}),}\\[2mm]
 {\displaystyle J_4^{(k)}=B(\mathbf{A}^{k-\frac{1}{2}}; I_{h}\overline{\Psi}^{k},\theta_{\Psi}^{k}-\theta_{\Psi}^{k-1})-B(\overline{\mathbf{A}}^{k}_{h}; I_{h}\overline{\Psi}^{k},\theta_{\Psi}^{k}-\theta_{\Psi}^{k-1}).}
 \end{array}
 \end{equation*}

By virtue of (\ref{eq:4-2}), we get
\begin{equation}\label{eq:4-17}
\begin{array}{@{}l@{}}
{\displaystyle \sum_{k=1}^{M} J_1^{(k)}=2\mathrm{i}\sum_{k=1}^{M}\left(\partial I_{h}\Psi^{k}-(\Psi_{t})^{k-\frac{1}{2}},\theta_{\Psi}^{k}-{\theta_{\Psi}^{k-1}}\right)}\\[2mm]
{\displaystyle \quad= 2\mathrm{i}\left(\partial I_{h}\Psi^{M}-(\Psi_{t})^{M-\frac{1}{2}},\theta_{\Psi}^{M}\right)-2\mathrm{i}\left(\partial I_{h}\Psi^{1}
-(\Psi_{t})^{\frac{1}{2}},\theta_{\Psi}^{0}\right)}\\[2mm]
{\displaystyle \quad-2\mathrm{i}\sum_{k=1}^{M-1}\left(\partial I_{h}\Psi^{k+1}
-\partial I_{h}\Psi^{k}-(\Psi_{t})^{k+\frac{1}{2}}
+(\Psi_{t})^{k-\frac{1}{2}},\,\theta_{\Psi}^{k}\right).}
\end{array}
\end{equation}

It follows from (\ref{eq:4-1}) and (\ref{eq:4-17})  that
\begin{equation}\label{eq:4-19}
\begin{array}{@{}l@{}}
{\displaystyle |\sum_{k=1}^{M} J_1^{(k)}|\leq C\big(h^{2r+2}+(\Delta t)^{4}\big)
+ C\|\theta_{\Psi}^{M}\|_{\mathcal{L}^2}^{2}+C\Delta t \sum_{k=1}^{M-1}{\|\theta_{\Psi}^{k}\|_{\mathcal{L}^2}^{2}}.}
\end{array}
\end{equation}

To estimate the term $ J_2^{(k)}$, we rewrite it as
\begin{equation*}
\begin{array}{@{}l@{}}
{\displaystyle 2 V_0\left(\Psi^{k-\frac{1}{2}}-\overline{\Psi}_{h}^{k},\theta_{\Psi}^{k}-\theta_{\Psi}^{k-1}\right) = 2V_0\left(\Psi^{k-\frac{1}{2}}
-I_{h}\overline{\Psi}^{k},\theta_{\Psi}^{k}-\theta_{\Psi}^{k-1}\right)}\\[2mm]
{\displaystyle \quad-2V_0\left(\frac{1}{2}(\theta_{\Psi}^{k}+\theta_{\Psi}^{k-1}),\theta_{\Psi}^{k}-\theta_{\Psi}^{k-1}\right)
\stackrel{\mathrm{def}}{=}J_2^{(k),1}+J_2^{(k),2}.}
\end{array}
\end{equation*}

By applying a standard argument, we find that
\begin{equation}\label{eq:4-22}
\begin{array}{@{}l@{}}
{\displaystyle  |\mathrm{Re}\big(\sum_{k=1}^{M} J_2^{(k)}\big)|\leq |\sum_{k=1}^{M}J_2^{(k),1}| +|\mathrm{Re}\big(\sum_{k=1}^{M} J_2^{(k),2}\big)|}\\[2mm]
{\displaystyle \quad \leq C\big(h^{2r+2}+(\Delta t)^{4}\big)
+C \|\theta_{\Psi}^{M}\|_{\mathcal{L}^2}^{2}+ C\Delta t \sum_{k=1}^{M-1}{\|\theta_{\Psi}^{k}\|_{\mathcal{L}^2}^{2}}.}
\end{array}
\end{equation}

We recall (\ref{eq:3-10}) and  rewrite $ J_3^{(k)} $ as follows:
\begin{equation}\label{eq:4-23}
\begin{array}{@{}l@{}}
{\displaystyle J_3^{(k)}= \left(\nabla( \Psi^{k-\frac{1}{2}}-I_{h}\overline{\Psi}^{k}),\;\nabla
(\theta_{\Psi}^{k}-\theta_{\Psi}^{k-1})\right)}\\[2mm]
{\displaystyle\quad \qquad +\left(|\mathbf{A}^{k-\frac{1}{2}}|^2(\Psi^{k-\frac{1}{2}}-I_{h}\overline{\Psi}^{k}),\theta_{\Psi}^{k}-\theta_{\Psi}^{k-1}\right)}\\[2mm]
{\displaystyle \quad \qquad+ \mathrm{i}\left(\nabla(\Psi^{k-\frac{1}{2}}-I_{h}\overline{\Psi}^{k})\mathbf{A}^{k-\frac{1}{2}},\;\theta_{\Psi}^{k}-\theta_{\Psi}^{k-1}\right)}\\[2mm]
{\displaystyle \quad \qquad-\mathrm{i}\left((\Psi^{k-\frac{1}{2}}-I_{h}\overline{\Psi}^{k})\mathbf{A}^{k-\frac{1}{2}},\;\nabla \theta_{\Psi}^{k}
-\nabla \theta_{\Psi}^{k-1}\right).}\\[2mm]
\end{array}
\end{equation}

By employing (\ref{eq:4-1}), (\ref{eq:4-2}), the regularity assumption (\ref{eq:2-10}) and Young's inequality,  we can prove the following estimate of  $ \sum_{k=1}^{M}{J_3^{(k)}} $ .
\begin{equation}\label{eq:4-23-0}
{\displaystyle  |\sum_{k=1}^{M}{J_3^{(k)}}| \leq C(h^{2r}+(\Delta t)^{4}) +C \Vert\theta_{\Psi}^{M}\Vert_{\mathcal{L}^2}^{2}
+ \frac{1}{16} \Vert\nabla\theta_{\Psi}^{M}\Vert_{\mathbf{L}^2}^{2}+C\Delta t \sum_{k=1}^{M-1}{\Vert\nabla\theta_{\Psi}^{k}\Vert_{\mathbf{L}^2}^{2}}.}
\end{equation}

The proof is standard but tedious. Due to space limitations, we omit it here.

In order to estimate $\sum\limits_{k=1}^{M} J_4^{(k)}$, we rewrite $ J_4^{(k)}$ in the following form:
\begin{equation}\label{eq:4-30}
\begin{array}{@{}l@{}}
{\displaystyle   J_4^{(k)}= \Big[B(\mathbf{A}^{k-\frac{1}{2}}; I_{h}\overline{\Psi}^{k},\theta_{\Psi}^{k}-\theta_{\Psi}^{k-1})
-B(\overline{\mathbf{A}}^{k}; I_{h}\overline{\Psi}^{k},\theta_{\Psi}^{k}-\theta_{\Psi}^{k-1})\Big]}\\[2mm]
{\displaystyle  \quad\quad +\Big[B(\overline{\mathbf{A}}^{k}; I_{h}\overline{\Psi}^{k},\theta_{\Psi}^{k}-\theta_{\Psi}^{k-1})-B({\bm \pi}_{h}\overline{\mathbf{A}}^{k}; I_{h}\overline{\Psi}^{k},\theta_{\Psi}^{k}-\theta_{\Psi}^{k-1})\Big]}\\[2mm]
{\displaystyle  \quad\quad +\Big[B({\bm \pi}_{h}\overline{\mathbf{A}}^{k}; I_{h}\overline{\Psi}^{k},\theta_{\Psi}^{k}-\theta_{\Psi}^{k-1})-B(\overline{\mathbf{A}}^{k}_{h}; I_{h}\overline{\Psi}^{k},\theta_{\Psi}^{k}-\theta_{\Psi}^{k-1})\Big]}\\[2mm]
{\displaystyle\quad\quad \stackrel{\mathrm{def}}{=}J_4^{(k),1}+J_4^{(k),2}+J_4^{(k),3}.}
\end{array}
\end{equation}

By applying (\ref{eq:3-15}) and (\ref{eq:4-2}), we deduce
\begin{equation}\label{eq:4-34}
\begin{array}{@{}l@{}}
{\displaystyle  |\sum_{k=1}^{M} J_4^{(k),1}| +  |\sum_{k=1}^{M} J_4^{(k),2}| \leq C\big(h^{2r}+(\Delta t)^{4}\big)
 +C \|\theta_{\Psi}^{M}\|_{\mathcal{L}^2}^{2}}\\[2mm]
{\displaystyle \quad+ \frac{1}{16} \|\nabla\theta_{\Psi}^{M}\|_{\mathbf{L}^2}^{2}+C\Delta t \sum_{k=1}^{M-1}{(\|\theta_{\Psi}^{k}\|_{\mathcal{L}^2}^{2}+\|\nabla\theta_{\Psi}^{k}\|_{\mathbf{L}^2}^{2})}.}
\end{array}
\end{equation}

In order to estimate $|\sum\limits_{k=1}^{M} J_4^{(k),3}|$, we rewrite it as follows.
\begin{equation}\label{eq:4-35}
\begin{array}{@{}l@{}}
{\displaystyle  \sum_{k=1}^{M}J_4^{(k),3}=\sum_{k=1}^{M}{ \left(I_{h}\overline{\Psi}^{k}({\bm \pi}_{h}\overline{\mathbf{A}}^{k}+\overline{\mathbf{A}}^{k}_{h})({\bm\pi}_{h}\overline{\mathbf{A}}^{k}
-\overline{\mathbf{A}}^{k}_{h}),\;\theta_{\Psi}^{k}-\theta_{\Psi}^{k-1}\right)}}\\[2mm]
{\displaystyle\quad\quad\quad\quad-\sum_{k=1}^{M}{\mathrm{i}\left( I_{h}\overline{\Psi}^{k}({\bm\pi}_{h}\overline{\mathbf{A}}^{k}-\overline{\mathbf{A}}^{k}_{h}),\;
\nabla\theta_{\Psi}^{k}-\nabla\theta_{\Psi}^{k-1}\right)}}\\[2mm]
{\displaystyle \quad\quad\quad\quad+\sum_{k=1}^{M}{\mathrm{i}\left(\nabla I_{h}\overline{\Psi}^{k}({\bm\pi}_{h}\overline{\mathbf{A}}^{k}
-\overline{\mathbf{A}}^{k}_{h}),\;\theta_{\Psi}^{k}-\theta_{\Psi}^{k-1}\right)}}\\[2mm]
{\displaystyle\quad\quad\quad \stackrel{\mathrm{def}}{=}Q_1+Q_2+Q_3.}
\end{array}
\end{equation}

Note that
\begin{equation}\label{eq:4-36}
\begin{array}{@{}l@{}}
{\displaystyle Q_1= \sum_{k=1}^{M}\left(I_{h}\overline{\Psi}^{k}({\bm\pi}_{h}\overline{\mathbf{A}}^{k}+\overline{\mathbf{A}}^{k}_{h})
({\bm\pi}_{h}\overline{\mathbf{A}}^{k}-\overline{\mathbf{A}}^{k}_{h}),\;\theta_{\Psi}^{k}-\theta_{\Psi}^{k-1}\right)}\\[2mm]
{\displaystyle \quad = -\left(I_{h}\overline{\Psi}^{M}({\bm\pi}_{h}\overline{\mathbf{A}}^{M}
+\overline{\mathbf{A}}^{M}_{h})
\overline{\theta}_{\mathbf{A}}^{M},\;\theta_{\Psi}^{M}\right) + \left(I_{h}\overline{\Psi}^{0}({\bm\pi}_{h}\overline{\mathbf{A}}^{0}+\overline{\mathbf{A}}^{0}_{h})
\overline{\theta}_{\mathbf{A}}^{0},\;\theta_{\Psi}^{0}\right)}\\[2mm]
{\displaystyle \quad + \sum_{k=1}^{M}\left(I_{h}\overline{\Psi}^{k}({\bm\pi}_{h}\overline{\mathbf{A}}^{k}
+\overline{\mathbf{A}}^{k}_{h})
\overline{\theta}_{\mathbf{A}}^{k}-I_{h}\overline{\Psi}^{k-1}({\bm\pi}_{h}\overline{\mathbf{A}}^{k-1}
+\overline{\mathbf{A}}^{k-1}_{h})
\overline{\theta}_{\mathbf{A}}^{k-1},\;\theta_{\Psi}^{k-1}\right).}
\end{array}
\end{equation}

By applying the Young's inequality and (\ref{eq:3-15}), we can estimate the first two terms on the right side of (\ref{eq:4-36}) as follows.
\begin{equation}\label{eq:4-37}
\begin{array}{@{}l@{}}
{\displaystyle |\left(I_{h}\overline{\Psi}^{M}({\bm\pi}_{h}\overline{\mathbf{A}}^{M}
+\overline{\mathbf{A}}^{M}_{h})\overline{\theta}_{\mathbf{A}}^{M},\;\theta_{\Psi}^{M}\right)|
+|\left(I_{h}\overline{\Psi}^{0}({\bm\pi}_{h}\overline{\mathbf{A}}^{0}
+\overline{\mathbf{A}}^{0}_{h})\overline{\theta}_{\mathbf{A}}^{0},\;\theta_{\Psi}^{0}\right)|}\\[2mm]
{\displaystyle \quad\leq \frac{1}{16}D(\overline{\theta}_{\mathbf{A}}^{M}, \overline{\theta}_{\mathbf{A}}^{M})+ C\|\theta_{\Psi}^{M}\|_{\mathcal{L}^2}^{2}+ Ch^{2r}.}
\end{array}
\end{equation}

From (\ref{eq:3-14}) and (\ref{eq:3-15}), we further deduce
\begin{equation}\label{eq:4-39}
\begin{array}{@{}l@{}}
{\displaystyle |\left(I_{h}\overline{\Psi}^{k}({\bm\pi}_{h}\overline{\mathbf{A}}^{k}
+\overline{\mathbf{A}}^{k}_{h})\overline{\theta}_{\mathbf{A}}^{k}-I_{h}\overline{\Psi}^{k-1}
({\bm\pi}_{h}\overline{\mathbf{A}}^{k-1}+\overline{\mathbf{A}}^{k-1}_{h})
\overline{\theta}_{\mathbf{A}}^{k-1},\;\theta_{\Psi}^{k-1}\right)|}\\[2mm]
{\displaystyle \quad\leq \Delta t \|I_{h}\overline{\Psi}^{k}\|_{\mathcal{L}^6}\| {\bm\pi}_{h}\overline{\mathbf{A}}^{k}+\overline{\mathbf{A}}^{k}_{h}\|_{\mathbf{L}^6}
\|\frac{1}{\Delta t} (\overline{\theta}_{\mathbf{A}}^{k}-\overline{\theta}_{\mathbf{A}}^{k-1})\|_{\mathbf{L}^2}\|\theta_{\Psi}^{k-1}\|_{\mathcal{L}^6}}\\[2mm]
{\displaystyle \quad \quad+ \Delta t \|\frac{I_{h}\overline{\Psi}^{k}-I_{h}\overline{\Psi}^{k-1}}{\Delta t}\|_{\mathcal{L}^2}\| {\bm\pi}_{h}\overline{\mathbf{A}}^{k}+\overline{\mathbf{A}}^{k}_{h}\|_{\mathbf{L}^6}
\|\overline{\theta}_{\mathbf{A}}^{k-1}\|_{\mathbf{L}^6}\|\theta_{\Psi}^{k-1}\|_{\mathcal{L}^6}}\\[2mm]
{\displaystyle \quad\quad +\Delta t \|I_{h}\overline{\Psi}^{k-1}\|_{\mathcal{L}^6}\| \overline{\theta}_{\mathbf{A}}^{k-1}\|_{\mathbf{L}^6}\|\frac{{\bm\pi}_{h}\overline{\mathbf{A}}^{k}
-{\bm\pi}_{h}\overline{\mathbf{A}}^{k-1}}{\Delta t}+\frac{\overline{\mathbf{A}}_{h}^{k}-\overline{\mathbf{A}}_{h}^{k-1}}{\Delta t}\|_{\mathbf{L}^2}\|\theta_{\Psi}^{k-1}\|_{\mathcal{L}^6}}\\[2mm]
{\displaystyle \quad\leq C\Delta t\Big\{\|\partial {\theta}_{\mathbf{A}}^{k}\|_{\mathbf{L}^2}^{2}
+ \|\partial {\theta}_{\mathbf{A}}^{k-1}\|_{\mathbf{L}^2}^{2}
+ D({\theta}_{\mathbf{A}}^{k-1},{\theta}_{\mathbf{A}}^{k-1}) + +D({\theta}_{\mathbf{A}}^{k-2},{\theta}_{\mathbf{A}}^{k-2})
+\|\nabla\theta_{\Psi}^{k-1}\|_{\mathbf{L}^2}^{2}\Big\},}\\[2mm]
\end{array}
\end{equation}
where we have used the fact:
\begin{equation}
\Vert \partial I_{h}\overline{\Psi}^{k} \Vert_{\mathcal{L}^{2}} \leq \Vert \partial \overline{\Psi}^{k} \Vert_{\mathcal{L}^{2}} + \Vert I_{h}\partial \overline{\Psi}^{k}  - \partial \overline{\Psi}^{k}\Vert_{\mathcal{L}^{2}} \leq C.
\end{equation}

Hence we get the following estimate:
\begin{equation}\label{eq:4-40}
\begin{array}{@{}l@{}}
{\displaystyle |Q_1|= |\sum_{k=1}^{M}\left(I_{h}\overline{\Psi}^{k}({\bm\pi}_{h}\overline{\mathbf{A}}^{k}
+\overline{\mathbf{A}}^{k}_{h})({\bm\pi}_{h}\overline{\mathbf{A}}^{k}-\overline{\mathbf{A}}^{k}_{h}),\;\theta_{\Psi}^{k}
-\theta_{\Psi}^{k-1}\right)|}\\[2mm]
{\displaystyle \quad\quad\quad\leq \frac{1}{16}D(\overline{\theta}_{\mathbf{A}}^{M}, \overline{\theta}_{\mathbf{A}}^{M})
+ C\|\theta_{\Psi}^{M}\|_{\mathcal{L}^2}^{2}+ C h^{2r} }\\[2mm]
{\displaystyle \quad\quad\quad + C\Delta t \sum_{k=1}^{M}\left(\|\partial {\theta}_{\mathbf{A}}^{k}\|_{\mathbf{L}^2}^{2}
+ D({\theta}_{\mathbf{A}}^{k},{\theta}_{\mathbf{A}}^{k})+ \|\nabla\theta_{\Psi}^{k}\|_{\mathbf{L}^2}^{2}\right).}
\end{array}
\end{equation}

Employing (\ref{eq:4-1}) and integrating by parts, we discover
\begin{equation}\label{eq:4-41}
\begin{array}{@{}l@{}}
{\displaystyle \mathrm{i}Q_2=\left(\nabla I_{h}\overline{\Psi}^{M}\overline{\theta}_{\mathbf{A}}^{M},\;\theta_{\Psi}^{M}\right)
 + \left(I_{h}\overline{\Psi}^{M}\nabla\cdot\overline{\theta}_{\mathbf{A}}^{M},\;\theta_{\Psi}^{M}\right)
 +\left(I_{h}\overline{\Psi}^{0}\overline{\theta}_{\mathbf{A}}^{0},\;\nabla\theta_{\Psi}^{0}\right)}\\[2mm]
{\displaystyle \quad\quad +\sum_{k=1}^{M}\left(I_{h}\overline{\Psi}^{k}\overline{\theta}_{\mathbf{A}}^{k}- I_{h}\overline{\Psi}^{k-1}\overline{\theta}_{\mathbf{A}}^{k-1},\;\nabla\theta_{\Psi}^{k-1}\right),}
\end{array}
\end{equation}

By using the Young's inequality, we can estimate the first three terms on the right side of (\ref{eq:4-41}) as follows:
 \begin{equation}\label{eq:4-42}
 \begin{array}{@{}l@{}}
 {\displaystyle |\left(\nabla I_{h}\overline{\Psi}^{M}\overline{\theta}_{\mathbf{A}}^{M},\;\theta_{\Psi}^{M}\right)|
 +|\left(I_{h}\overline{\Psi}^{M}\nabla\cdot\overline{\theta}_{\mathbf{A}}^{M},\;\theta_{\Psi}^{M}\right)|
 +|\left(I_{h}\overline{\Psi}^{0}\overline{\theta}_{\mathbf{A}}^{0},\;\nabla\theta_{\Psi}^{0}\right)|}\\[2mm]
 {\displaystyle \quad\leq \frac{1}{16}D(\overline{\theta}_{\mathbf{A}}^{M}, \overline{\theta}_{\mathbf{A}}^{M})+ C\|\theta_{\Psi}^{M}\|_{\mathcal{L}^2}^{2}+C h^{2r}.}
 \end{array}
 \end{equation}

 Using (\ref{eq:4-1}), the last term on the right side of (\ref{eq:4-41}) can be estimated by
 \begin{equation}\label{eq:4-43}
 \begin{array}{@{}l@{}}
 {\displaystyle |\sum_{k=1}^{M}\left(I_{h}\overline{\Psi}^{k}\overline{\theta}_{\mathbf{A}}^{k}- I_{h}\overline{\Psi}^{k-1}\overline{\theta}_{\mathbf{A}}^{k-1},\;\nabla\theta_{\Psi}^{k-1}\right)|}\\[2mm]
{\displaystyle \quad\leq  C\Delta t \sum_{k=1}^{M}\left(\|\overline{\theta}_{\mathbf{A}}^{k}\|_{\mathbf{H}^{1}}
\|\nabla\theta_{\Psi}^{k-1}\|_{\mathbf{L}^2}+\|\partial \overline{\theta}_{\mathbf{A}}^{k}\|_{\mathbf{L}^2}\|\nabla\theta_{\Psi}^{k-1}\|_{\mathbf{L}^2}\right)}\\[2mm]
{\displaystyle \quad\leq  C\Delta t \sum_{k=0}^{M}\left(D({\theta}_{\mathbf{A}}^{k},{\theta}_{\mathbf{A}}^{k})
+\|\partial {\theta}_{\mathbf{A}}^{k}\|_{\mathbf{L}^2}^{2}+\|\nabla\theta_{\Psi}^{k}\|_{\mathbf{L}^2}^{2}\right).}
 \end{array}
 \end{equation}

Hence we get
 \begin{equation}\label{eq:4-44}
 \begin{array}{@{}l@{}}
 {\displaystyle |Q_2|=|-\mathrm{i}\sum_{k=1}^{M}\left(I_{h}\overline{\Psi}^{k}({\bm\pi}_{h}\overline{\mathbf{A}}^{k}-\overline{\mathbf{A}}^{k}_{h}),\;
 \nabla\theta_{\Psi}^{k}-\nabla\theta_{\Psi}^{k-1}\right)|}\\[2mm]
 {\displaystyle \quad\quad\leq \frac{1}{16}D(\overline{\theta}_{\mathbf{A}}^{M}, \overline{\theta}_{\mathbf{A}}^{M})
 + C\|\theta_{\Psi}^{M}\|_{\mathcal{L}^2}^{2}+C h^{2r}}\\[2mm]
 {\displaystyle \quad \quad\quad+ C\Delta t \sum_{k=0}^{M}\left(D({\theta}_{\mathbf{A}}^{k},{\theta}_{\mathbf{A}}^{k})+\|\partial {\theta}_{\mathbf{A}}^{k}\|_{\mathbf{L}^2}^{2}+\|\nabla\theta_{\Psi}^{k}\|_{\mathbf{L}^2}^{2}\right).}
 \end{array}
 \end{equation}

Reasoning as before, we can estimate $ Q_3 $ as follows:
 \begin{equation}\label{eq:4-45}
 \begin{array}{@{}l@{}}
 {\displaystyle |Q_3|=|\mathrm{i}\sum_{k=1}^{M}\left(\nabla I_{h}\overline{\Psi}^{k}({\bm\pi}_{h}\overline{\mathbf{A}}^{k}
 -\overline{\mathbf{A}}^{k}_{h}),\;\theta_{\Psi}^{k}-\theta_{\Psi}^{k-1}\right)|}\\[2mm]
 {\displaystyle \quad\quad\leq  \frac{1}{16}D(\overline{\theta}_{\mathbf{A}}^{M}, \overline{\theta}_{\mathbf{A}}^{M}) + C\|\theta_{\Psi}^{M}\|_{\mathcal{L}^2}^{2}+C h^{2r}}\\[2mm]
 {\displaystyle \quad\quad\quad + C\Delta t \sum_{k=0}^{M}\left(D({\theta}_{\mathbf{A}}^{k},{\theta}_{\mathbf{A}}^{k})+\|\partial {\theta}_{\mathbf{A}}^{k}\|_{\mathbf{L}^2}^{2}+\|\nabla\theta_{\Psi}^{k}\|_{\mathbf{L}^2}^{2}\right).}
 \end{array}
 \end{equation}

Combining (\ref{eq:4-40}), (\ref{eq:4-44}) and (\ref{eq:4-45}) implies
\begin{equation}\label{eq:4-46}
\begin{array}{@{}l@{}}
{\displaystyle |\sum_{k=1}^{M} J_4^{(k),3}| \leq  \frac{3}{16}D(\overline{\theta}_{\mathbf{A}}^{M}, \overline{\theta}_{\mathbf{A}}^{M})
+ C\|\theta_{\Psi}^{M}\|_{\mathcal{L}^2}^{2}+C h^{2r}}\\[2mm]
 {\displaystyle \quad\quad\quad + C\Delta t \sum_{k=0}^{M}\left(D({\theta}_{\mathbf{A}}^{k},{\theta}_{\mathbf{A}}^{k})
 +\|\partial {\theta}_{\mathbf{A}}^{k}\|_{\mathbf{L}^2}^{2}
 +\|\nabla\theta_{\Psi}^{k}\|_{\mathbf{L}^2}^{2}\right).}
 \end{array}
 \end{equation}

It follows from (\ref{eq:4-30}), (\ref{eq:4-34}) and (\ref{eq:4-46}) that
 \begin{equation}\label{eq:4-47}
 \begin{array}{@{}l@{}}
 {\displaystyle \displaystyle |\sum_{k=1}^{M} J_4^{(k)}|\leq C\left(h^{2r}+(\Delta t)^{4}\right)+\frac{3}{16}D(\overline{\theta}_{\mathbf{A}}^{M}, \overline{\theta}_{\mathbf{A}}^{M}) +\frac{1}{16} \|\nabla\theta_{\Psi}^{M}\|_{\mathbf{L}^2}^{2}
 +C\|\theta_{\Psi}^{M}\|_{\mathcal{L}^2}^{2} }\\[2mm]
 {\displaystyle \qquad+C\Delta t \sum_{k=0}^{M}\Big\{D({\theta}_{\mathbf{A}}^{k},{\theta}_{\mathbf{A}}^{k}) +\|\partial {\theta}_{\mathbf{A}}^{k}\|_{\mathbf{L}^2}^{2}+\|\theta_{\Psi}^{k}\|_{\mathcal{L}^2}^{2}
 +\|\nabla\theta_{\Psi}^{k}\|_{\mathbf{L}^2}^{2}\Big\} }\\[2mm]
 \end{array}
 \end{equation}

Now take the real part of (\ref{eq:4-16}), and we get
\begin{equation}\label{eq:4-48}
\begin{array}{@{}l@{}}
{\displaystyle \Delta t \mathrm{Re}\left[B\left(\overline{\mathbf{A}}_{h}^{k};
\overline{\theta}_{\Psi}^{k},\partial \theta_{\Psi}^{k}\right)\right] =\mathrm{Re}
\big(J_1^{(k)}\big)+\mathrm{Re}\big(J_2^{(k)}\big)+\mathrm{Re}
\big(J_3^{(k)}\big)+\mathrm{Re}
\big(J_4^{(k)}\big).}
\end{array}
\end{equation}

Similarly to (\ref{eq:3-11}), we have
\begin{equation}\label{eq:4-49}
\begin{array}{@{}l@{}}
{\displaystyle \mathrm{Re}\left[B\left(\overline{\mathbf{A}}_{h}^{k};\overline{\theta}_{\Psi}^{k},\partial \theta_{\Psi}^{k}\right)\right]=
 -\left(\frac{1}{2}(\overline{\mathbf{A}}_{h}^{k}+\overline{\mathbf{A}}_{h}^{k-1})|\theta_{\Psi}^{k-1}|^{2},\frac{1}{2}(\partial \mathbf{A}_{h}^{k}+\partial \mathbf{A}_{h}^{k-1})\right)}\\[2mm]
{\displaystyle\quad +\frac{1}{2}\partial B(\overline{\mathbf{A}}_{h}^{k};\theta_{\Psi}^{k},\theta_{\Psi}^{k})-\left(f(\theta_{\Psi}^{k-1},\theta_{\Psi}^{k-1}),\frac{1}{2}(\partial \mathbf{A}_{h}^{k}+\partial \mathbf{A}_{h}^{k-1})\right).}
\end{array}
\end{equation}

Substituting (\ref{eq:4-49}) into (\ref{eq:4-48}) and summing over $k = 1,2,\cdots,M$, we get
\begin{equation}\label{eq:4-51}
\begin{array}{@{}l@{}}
{\displaystyle \frac{1}{2}B(\overline{\mathbf{A}}_{h}^{M};\theta_{\Psi}^{M},\theta_{\Psi}^{M})
=\sum_{k=1}^{M}\left[\mathrm{Re}
\big(J_1^{(k)}\big)+\mathrm{Re}\big(J_2^{(k)}\big)+\mathrm{Re}
\big(J_3^{(k)}\big)+\mathrm{Re}
\big(J_4^{(k)}\big)\right]}\\[2mm]
{\displaystyle\quad+ \frac{1}{2}B(\overline{\mathbf{A}}_{h}^{0};\theta_{\Psi}^{0},\theta_{\Psi}^{0})+\Delta t\sum_{k=1}^{M} \Big(\frac{1}{2}(\overline{\mathbf{A}}_{h}^{k}+\overline{\mathbf{A}}_{h}^{k-1})|\theta_{\Psi}^{k-1}|^{2},\frac{1}{2}(\partial \mathbf{A}_{h}^{k}+\partial \mathbf{A}_{h}^{k-1})\Big)}\\[2mm]
{\displaystyle \quad+\Delta t\sum_{k=1}^{M}\Big(f(\theta_{\Psi}^{k-1},\theta_{\Psi}^{k-1}),\frac{1}{2}(\partial \mathbf{A}_{h}^{k}+\partial \mathbf{A}_{h}^{k-1})\Big).}
\end{array}
\end{equation}

Combining (\ref{eq:4-19}), (\ref{eq:4-22}), (\ref{eq:4-23-0}) and (\ref{eq:4-47}) implies
\begin{equation}\label{eq:4-53}
\begin{array}{@{}l@{}}
{\displaystyle
\sum_{k=1}^{M}\left[\mathrm{Re}
\big(J_1^{(k)}\big)+\mathrm{Re}\big(J_2^{(k)}\big)+\mathrm{Re}
\big(J_3^{(k)}\big)+\mathrm{Re}
\big(J_4^{(k)}\big)\right]}\\[2mm]
{\displaystyle\quad \leq |\sum_{k=1}^{M}J_1^{(k)}|+|\sum_{k=1}^{M}J_2^{(k)}|+|\sum_{k=1}^{M}J_3^{(k)}|+|\sum_{k=1}^{M}J_4^{(k)}|}\\[2mm]
{\displaystyle \leq  C\left(h^{2r}+(\Delta t)^{4}\right)+\frac{3}{16}D(\overline{\theta}_{\mathbf{A}}^{M}, \overline{\theta}_{\mathbf{A}}^{M})
+\frac{1}{8} \|\nabla\theta_{\Psi}^{M}\|_{\mathbf{L}^2}^{2} +C\|\theta_{\Psi}^{M}\|_{\mathcal{L}^2}^{2}}\\[2mm]
 {\displaystyle \qquad+C\Delta t \sum_{k=0}^{M}\left(D({\theta}_{\mathbf{A}}^{k},{\theta}_{\mathbf{A}}^{k})
 +\|\partial {\theta}_{\mathbf{A}}^{k}\|_{\mathbf{L}^2}^{2}+\|\theta_{\Psi}^{k}\|_{\mathcal{L}^2}^{2}
 +\|\nabla\theta_{\Psi}^{k}\|_{\mathbf{L}^2}^{2}\right).}
 \end{array}
 \end{equation}
 
By employing Theorem~\ref{thm3-1}, we discover
\begin{equation}\label{eq:4-53-0}
\Big(\frac{1}{2}(\overline{\mathbf{A}}_{h}^{k}+\overline{\mathbf{A}}_{h}^{k-1})|\theta_{\Psi}^{k-1}|^{2},\frac{1}{2}(\partial \mathbf{A}_{h}^{k}+\partial \mathbf{A}_{h}^{k-1})\Big) \leq C \|\nabla\theta_{\Psi}^{k-1}\|_{\mathbf{L}^2}^{2}.
\end{equation} 

Setting
\begin{equation}\label{eq:4-54}
J_5^{(k)} = \Big(f(\theta_{\Psi}^{k-1},\theta_{\Psi}^{k-1}),\frac{1}{2}(\partial \mathbf{A}_{h}^{k}+\partial \mathbf{A}_{h}^{k-1})\Big),
\end{equation}
 and applying (\ref{eq:4-53}) and (\ref{eq:4-53-0}), we obtain
 \begin{equation}\label{eq:4-55}
 \begin{array}{@{}l@{}}
 {\displaystyle \frac{1}{2}B(\overline{\mathbf{A}}_{h}^{M};\theta_{\Psi}^{M},\theta_{\Psi}^{M}) \leq C\left(h^{2r}+(\Delta t)^{4}\right)+\frac{3}{16}D(\overline{\theta}_{\mathbf{A}}^{M},\overline{\theta}_{\mathbf{A}}^{M}) +\frac{1}{8} \|\nabla\theta_{\Psi}^{M}\|_{\mathbf{L}^2}^{2}
  +C\|\theta_{\Psi}^{M}\|_{\mathcal{L}^2}^{2}}\\[2mm]
  {\displaystyle \qquad +\Delta t \sum_{k=1}^{M} J_5^{(k)} + C\Delta t \sum_{k=0}^{M}\left\{D({\theta}_{\mathbf{A}}^{k},{\theta}_{\mathbf{A}}^{k})
  +\|\partial {\theta}_{\mathbf{A}}^{k}\|_{\mathbf{L}^2}^{2}+\|\theta_{\Psi}^{k}\|_{\mathcal{L}^2}^{2}
  +\|\nabla\theta_{\Psi}^{k}\|_{\mathbf{L}^2}^{2}\right\}.}
 \end{array}
 \end{equation}

Arguing as in the proof of Theorem~\ref{thm3-1}, we discover  
 \begin{equation}\label{eq:4-56}
 \|\left(\mathrm{i}\nabla
+\overline{\mathbf{A}}^{M}_{h}\right)\theta^{M}_{\Psi}\|_{\mathbf{L}^2}\geq \|\nabla\theta_{\Psi}^{M}\|_{\mathbf{L}^2} - \|\overline{\mathbf{A}}_{h}^{M}\theta_{\Psi}^{M}\|_{\mathbf{L}^2} \geq  \frac{3}{4}\|\nabla\theta_{\Psi}^{M}\|_{\mathbf{L}^2} - C\|\theta_{\Psi}^{M}\|_{\mathbf{L}^2}, 
 \end{equation}
and thus
 \begin{equation}\label{eq:4-57}
B(\overline{\mathbf{A}}_{h}^{M};\theta_{\Psi}^{M},\theta_{\Psi}^{M}) + C\|\theta_{\Psi}^{M}\|_{\mathbf{L}^2}^{2} \geq \frac{9}{32}\|\nabla\theta_{\Psi}^{M}\|_{\mathbf{L}^2}^{2}.
\end{equation}

Substituting (\ref{eq:4-57}) into (\ref{eq:4-55}) , we obtain
 \begin{equation}\label{eq:4-59}
 \begin{array}{@{}l@{}}
 {\displaystyle \frac{1}{64}\|\nabla\theta_{\Psi}^{M}\|_{\mathbf{L}^2}^{2}
 \leq C\left(h^{2r}+(\Delta t)^{4}\right)+\frac{3}{16}D(\overline{\theta}_{\mathbf{A}}^{M},\overline{\theta}_{\mathbf{A}}^{M})  +C\|\theta_{\Psi}^{M}\|_{\mathcal{L}^2}^{2}+\Delta t \sum_{k=1}^{M} J_5^{(k)}}\\[2mm]
 {\displaystyle \quad\quad\quad+C\Delta t \sum_{k=0}^{M}\left\{D({\theta}_{\mathbf{A}}^{k},{\theta}_{\mathbf{A}}^{k})
 +\|\partial {\theta}_{\mathbf{A}}^{k}\|_{\mathbf{L}^2}^{2}+\|\theta_{\Psi}^{k}\|_{\mathcal{L}^2}^{2}
 +\|\nabla\theta_{\Psi}^{k}\|_{\mathbf{L}^2}^{2}\right\}.}
 \end{array}
 \end{equation}

 Multiplying (\ref{eq:4-15}) with $(C + 1)$ and adding to (\ref{eq:4-59}), we end up with
 \begin{equation}\label{eq:4-60}
 \begin{array}{@{}l@{}}
 {\displaystyle \|\theta_{\Psi}^{M}\|_{\mathcal{L}^2}^{2}+\frac{1}{64}\|\nabla\theta_{\Psi}^{M}\|_{\mathbf{L}^2}^{2}
 \leq C\left(h^{2r}+(\Delta t)^{4}\right) +\frac{3}{16}D(\overline{\theta}_{\mathbf{A}}^{M},\overline{\theta}_{\mathbf{A}}^{M})+\Delta t \sum_{k=1}^{M} J_5^{(k)}}\\[2mm]
{\displaystyle \quad\quad\quad+C\Delta t \sum_{k=0}^{M}\left\{D({\theta}_{\mathbf{A}}^{k},{\theta}_{\mathbf{A}}^{k})
+\|\partial {\theta}_{\mathbf{A}}^{k}\|_{\mathbf{L}^2}^{2}+\|\theta_{\Psi}^{k}\|_{\mathcal{L}^2}^{2}
+\|\nabla\theta_{\Psi}^{k}\|_{\mathbf{L}^2}^{2}\right\}.}
 \end{array}
 \end{equation}

\subsection{Estimates for (\ref{eq:4-4})}
Setting
\begin{equation}\label{eq:4-61}
\begin{array}{@{}l@{}}
{\displaystyle  K_1^{(k)}=\left((\mathbf{A}_{tt})^{k-1}-\partial^{2}{\bm\pi}_{h}\mathbf{A}^{k},\mathbf{v}\right),}\\[2mm]
{\displaystyle K_2^{(k)}= D(\mathbf{A}^{k-1}-\widetilde{{\bm\pi}_{h}\mathbf{A}^{k}},\mathbf{v}),}\\[2mm]
{\displaystyle K_3^{(k)} = \Big(|\Psi^{k-1}|^{2}\mathbf{A}^{k-1}-|\Psi^{k-1}_{h}|^{2}\frac{\overline{\mathbf{A}}^{k}_{h}
+\overline{\mathbf{A}}^{k-1}_{h}}{2},\;\mathbf{v}\Big),}\\[2mm]
{\displaystyle K_4^{(k)}= \left(f(\Psi^{k-1},\Psi^{k-1})-f(\Psi^{k-1}_{h},\Psi^{k-1}_{h}),\;\mathbf{v}\right),}
\end{array}
\end{equation}
we rewrite (\ref{eq:4-4}) as follows:
\begin{equation}\label{eq:4-62}
{\displaystyle \left(\partial^{2}\theta^{k}_{\mathbf{A}},\mathbf{v}\right)+D(\widetilde{\theta^{k}_{\mathbf{A}}},\mathbf{v})=
 K_1^{(k)}+ K_2^{(k)}+ K_3^{(k)}+ K_4^{(k)}.}
\end{equation}

 We first estimate  $\sum\limits_{k=1}^{M} K_1^{(k)}$, $\sum\limits_{k=1}^{M} K_3^{(k)}$ and $\sum\limits_{k=1}^{M} K_4^{(k)}$.
 Under the regularity assumption of $\mathbf{A}$ in (\ref{eq:2-10}),  we have
\begin{equation}\label{eq:4-63}
\sum_{k=1}^{M} |K_1^{(k)}| \leq \frac{C}{\Delta t}\left(h^{2r+2}+(\Delta t)^{4}\right)
+ C\sum_{k=1}^{M} \|\mathbf{v}\|_{\mathbf{L}^2}^{2}.
\end{equation}

By applying the regularity assumption (\ref{eq:2-10}), the interpolation error estimates (\ref{eq:4-1}) and Theorem~\ref{thm3-1}, it is easy to deduce 
 \begin{equation}\label{eq:4-68}
 \sum_{k=1}^{M} K_3^{(k)} \leq \frac{C}{\Delta t}\left[h^{2r}+(\Delta t)^{4}\right] + C\sum_{k=0}^{M}\left(D({\theta}_{\mathbf{A}}^{k}, {\theta}_{\mathbf{A}}^{k})+\|\nabla\theta_{\Psi}^{k}\|_{\mathbf{L}^2}^{2} +\|\mathbf{v}\|_{\mathbf{L}^2}^{2}\right).
\end{equation}

In order to estimate  $\sum\limits_{k=1}^{M} K_4^{(k)}$, we rewrite $ K_4^{(k)}$ in the following form:
\begin{equation}\label{eq:4-69}
\begin{array}{@{}l@{}}
{\displaystyle K_4^{(k)}=\left(f(\Psi^{k-1},\Psi^{k-1})-f(I_{h}\Psi^{k-1}, I_{h}\Psi^{k-1}),\;\mathbf{v}\right) }\\[2mm]
{\displaystyle \quad +\left(f(I_{h}\Psi^{k-1}, I_{h}\Psi^{k-1})-f(\Psi^{k-1}_{h},\Psi^{k-1}_{h}),\;\mathbf{v}\right)\stackrel{\mathrm{def}}{=} K_4^{(k),1}+K_4^{(k),2}.}
\end{array}
\end{equation}

We observe that
\begin{equation}\label{eq:4-70}
\begin{array}{@{}l@{}}
{\displaystyle f(\psi,\psi)-f(\varphi,\varphi)=\frac{\mathrm{i}}{2}(\psi^{\ast}\nabla \psi- \psi \nabla\psi^{\ast})
  -\frac{\mathrm{i}}{2} (\varphi^{\ast} \nabla\varphi-\varphi \nabla\varphi^{\ast}) }\\[2mm]
{\displaystyle = -\frac{\mathrm{i}}{2}\left(\varphi^{\ast}\nabla(\varphi-\psi)-\varphi\nabla(\varphi-\psi)^{\ast}\right)
 +\frac{\mathrm{i}}{2}\left((\varphi-\psi)\nabla\psi^{\ast}-(\varphi-\psi)^{\ast}\nabla\psi\right).}
\end{array}
\end{equation}

We obtain from (\ref{eq:4-1}) and (\ref{eq:4-70}) that
\begin{equation}\label{eq:4-71}
\begin{array}{@{}l@{}}
{\displaystyle K_4^{(k),1}  \leq C\left\{h^{2r}+\|\mathbf{v}\|_{\mathbf{L}^2}^{2}\right\}.}
\end{array}
\end{equation}

Similarly, from (\ref{eq:4-1}) and (\ref{eq:4-70}), we deduce
\begin{equation}\label{eq:4-72}
\begin{array}{@{}l@{}}
{\displaystyle K_4^{(k),2}=\left(f(I_{h}\Psi^{k-1}, I_{h}\Psi^{k-1})-f(\Psi^{k-1}_{h},\Psi^{k-1}_{h}),\;\mathbf{v}\right)}\\[2mm]
{\displaystyle \quad =  -\frac{\mathrm{i}}{2}\left((\theta_{\Psi}^{k-1})^{\ast}\nabla\theta_{\Psi}^{k-1}
-\theta_{\Psi}^{k-1}\nabla(\theta_{\Psi}^{k-1})^{\ast},\;\mathbf{v}\right) }\\[2mm]
{\displaystyle \quad\quad-\frac{\mathrm{i}}{2}\left((I_{h}\Psi^{k-1})^{\ast}\nabla\theta_{\Psi}^{k-1}
-I_{h}\Psi^{k-1}\nabla(\theta_{\Psi}^{k-1})^{\ast},\;\mathbf{v}\right) }\\[2mm]
{\displaystyle\quad \quad +\frac{\mathrm{i}}{2}\left(\theta_{\Psi}^{k-1}\nabla (I_{h}\Psi^{k-1})^{\ast}-(\theta_{\Psi}^{k-1})^{\ast}\nabla I_{h}\Psi^{k-1},\;\mathbf{v}\right)}\\[2mm]
{\displaystyle \quad \leq -\left(f(\theta_{\Psi}^{k-1},\theta_{\Psi}^{k-1}),\mathbf{v}\right)
+ C \| I_{h}\Psi^{k-1}\|_{\mathcal{L}^{\infty}}\|\nabla\theta_{\Psi}^{k-1}\|_{\mathbf{L}^2}\|\mathbf{v}\|_{\mathbf{L}^2}}\\[2mm]
{\displaystyle \quad\quad+C\|\nabla I_{h}\Psi^{k-1}\|_{\mathbf{L}^{3}}\|\theta_{\Psi}^{k-1}\|_{\mathcal{L}^6}
\|\mathbf{v}\|_{\mathbf{L}^2}}\\[2mm]
{\displaystyle \quad \leq -\left(f(\theta_{\Psi}^{k-1},\theta_{\Psi}^{k-1}),\mathbf{v}\right)
+C\left(\|\nabla\theta_{\Psi}^{k-1}\|_{\mathbf{L}^2}^{2}
+\|\mathbf{v}\|_{\mathbf{L}^2}^{2}\right).}
\end{array}
\end{equation}

Combining  (\ref{eq:4-69}), (\ref{eq:4-71}) and (\ref{eq:4-72}) gives
\begin{equation}\label{eq:4-73}
\begin{array}{@{}l@{}}
{\displaystyle\sum_{k=1}^{M} K_4^{(k)}\leq \frac{C h^{2r}}{\Delta t} -\sum_{k=1}^{M}\left(f(\theta_{\Psi}^{k-1},\theta_{\Psi}^{k-1}),\mathbf{v}\right)
+C\sum_{k=1}^{M}\left(\|\nabla\theta_{\Psi}^{k-1}\|_{\mathbf{L}^2}^{2}
+\|\mathbf{v}\|_{\mathbf{L}^2}^{2}\right).}
\end{array}
\end{equation}

From (\ref{eq:4-63}), (\ref{eq:4-68}) and (\ref{eq:4-73}), we obtain
\begin{equation}\label{eq:4-74}
\begin{array}{@{}l@{}}
{\displaystyle \sum_{k=1}^{M} \Big(K_1^{k)}+K_3^{k)}+K_4^{k)}\Big)
\leq -\sum_{k=1}^{M}\left(f(\theta_{\Psi}^{k-1},
\theta_{\Psi}^{k-1}),\mathbf{v}\right)+\frac{C}{\Delta t}\left\{h^{2r}+(\Delta t)^{4}\right\}}\\[2mm]
{\displaystyle \qquad \quad+ C\sum_{k=0}^{M}\left(D({\theta}_{\mathbf{A}}^{k}, {\theta}_{\mathbf{A}}^{k})
+\|\nabla\theta_{\Psi}^{k}\|_{\mathbf{L}^2}^{2}+\|\mathbf{v}\|_{\mathbf{L}^2}^{2}\right).}
\end{array}
\end{equation}

Now  taking $\mathbf{v}= \frac{\displaystyle 1}{\displaystyle 2\Delta t}(\theta_{\mathbf{A}}^{k}
-\theta_{\mathbf{A}}^{k-2})=\frac{\displaystyle 1}{\displaystyle 2}(\partial \theta_{\mathbf{A}}^{k}
+\partial \theta_{\mathbf{A}}^{k-1})$ in (\ref{eq:4-62}), we find
\begin{equation}\label{eq:4-75}
\begin{array}{@{}l@{}}
{\displaystyle \Big(\partial^{2}\theta^{k}_{\mathbf{A}},\frac{1}{2}(\partial \theta_{\mathbf{A}}^{k}+\partial \theta_{\mathbf{A}}^{k-1})\Big)
+D\Big(\widetilde{\theta^{k}_{\mathbf{A}}},\frac{1}{2}(\partial \theta_{\mathbf{A}}^{k}+\partial \theta_{\mathbf{A}}^{k-1})\Big)}\\[2mm]
{\displaystyle \quad = \frac{1}{2\varDelta t}\left(\|\partial \theta_{\mathbf{A}}^{k}\|_{\mathbf{L}^2}^{2}
-\|\partial \theta_{\mathbf{A}}^{k-1}\|_{\mathbf{L}^2}^{2}\right)+\frac{1}{4\Delta t}\left(D({\theta}_{\mathbf{A}}^{k}, {\theta}_{\mathbf{A}}^{k})
-(D({\theta}_{\mathbf{A}}^{k-2}, {\theta}_{\mathbf{A}}^{k-2})\right)}\\[2mm]
{\displaystyle  \quad =  K_1^{(k)}+ K_2^{(k)}+ K_3^{(k)}+ K_4^{(k)}.}
\end{array}
\end{equation}

Note that 
\begin{equation*}
\mathbf{v} = \frac{1}{2}(\partial \theta_{\mathbf{A}}^{k}
+\partial \theta_{\mathbf{A}}^{k-1}) =\frac{1}{2}(\partial \mathbf{A}_{h}^{k}
+\partial \mathbf{A}_{h}^{k-1}) -  \frac{1}{2}(\partial {\bm \pi}_{h}{\mathbf{A}}^{k}
+\partial {\bm \pi}_{h}{\mathbf{A}}^{k-1}). 
\end{equation*}

Thus we have 
\begin{equation}\label{eq:4-75-0}
\begin{array}{@{}l@{}}
{\displaystyle  -\sum_{k=1}^{M}\left(f(\theta_{\Psi}^{k-1},
\theta_{\Psi}^{k-1}),\mathbf{v}\right) }\\[2mm]
{\displaystyle = -\sum_{k=1}^{M} J_5^{(k)} + \sum_{k=1}^{M}\left(f(\theta_{\Psi}^{k-1},
\theta_{\Psi}^{k-1}),\frac{\partial {\bm \pi}_{h}{\mathbf{A}}^{k}
+\partial {\bm \pi}_{h}{\mathbf{A}}^{k-1}}{2}\right) }\\[2mm]
{\displaystyle \leq  -\sum_{k=1}^{M} J_5^{(k)} + C\sum_{k=1}^{M}\|\nabla\theta_{\Psi}^{k-1}\|_{\mathbf{L}^{2}}^{2}}.
\end{array}
\end{equation}
Here we have used the definition of $ J_5^{(k)} $ in (\ref{eq:4-54}).

Multiplying (\ref{eq:4-75}) by $\Delta t $,  and using (\ref{eq:4-74}) and (\ref{eq:4-75-0}), we obtain
\begin{equation}\label{eq:4-76}
\begin{array}{@{}l@{}}
{\displaystyle  \frac{1}{2}\|\partial \theta_{\mathbf{A}}^{M}\|_{\mathbf{L}^2}^{2}
+\frac{1}{4}D({\theta}_{\mathbf{A}}^{M}, {\theta}_{\mathbf{A}}^{M})
+\frac{1}{4}D({\theta}_{\mathbf{A}}^{M-1}, {\theta}_{\mathbf{A}}^{M-1})}\\[2mm]
{\displaystyle \quad \leq C\left\{h^{2r}+(\Delta t)^{4} \right\}
+ C\Delta t\sum_{k=0}^{M}\left(D({\theta}_{\mathbf{A}}^{k}, {\theta}_{\mathbf{A}}^{k})+\|\nabla\theta_{\Psi}^{k}\|_{\mathbf{L}^2}^{2}
+ \|\partial \theta_{\mathbf{A}}^{k}\|_{\mathbf{L}^2}^{2} \right) }\\[2mm]
{\displaystyle \quad+ \Delta t \sum_{k=1}^{M} K_2^{(k)}
-\Delta t\sum_{k=1}^{M} J_5^{(k)}.}
\end{array}
\end{equation}

Since $\mathbf{v}=\frac{\displaystyle 1}{\displaystyle 2}(\partial \theta_{\mathbf{A}}^{k}+\partial \theta_{\mathbf{A}}^{k-1})$, by applying (\ref{eq:4-2}) and the Young's inequality, we get
\begin{equation}\label{eq:4-77}
\begin{array}{@{}l@{}}
{\displaystyle  \Delta t \sum_{k=1}^{M} K_2^{(k)}=\Delta t \sum_{k=1}^{M} D(\mathbf{A}^{k-1}
-\widetilde{{\bm\pi}_{h}\mathbf{A}^{k}}, \frac{1}{2}(\partial \theta_{\mathbf{A}}^{k}
+\partial \theta_{\mathbf{A}}^{k-1}))}\\[2mm]
{\displaystyle \leq C\left\{h^{2r} + (\Delta t)^{4}\right\} + \frac{1}{32} D\left(\theta_{\mathbf{A}}^{M},\theta_{\mathbf{A}}^{M}\right)
+\frac{1}{32} D\left(\theta_{\mathbf{A}}^{M-1},\theta_{\mathbf{A}}^{M-1}\right)
+ C\Delta t \sum_{k=0}^{M}D\left(\theta_{\mathbf{A}}^{k},\theta_{\mathbf{A}}^{k}\right)}\\[2mm]
\end{array}
\end{equation}

Combining (\ref{eq:4-60}), (\ref{eq:4-76}) and (\ref{eq:4-77}) implies
\begin{equation}\label{eq:4-81}
\begin{array}{@{}l@{}}
{\displaystyle \|\theta_{\Psi}^{M}\|_{\mathcal{L}^2}^{2}+\frac{1}{64}\|\nabla\theta_{\Psi}^{M}\|_{\mathbf{L}^2}^{2}
 +\frac{1}{2}\|\partial \theta_{\mathbf{A}}^{M}\|_{\mathbf{L}^2}^{2} + \frac{1}{8}D({\theta}_{\mathbf{A}}^{M}, {\theta}_{\mathbf{A}}^{M})
 +\frac{1}{8}D({\theta}_{\mathbf{A}}^{M-1}, {\theta}_{\mathbf{A}}^{M-1})}\\[2mm]
{\displaystyle \quad\leq  C\left\{h^{2r}+(\Delta t)^{4}\right\}
+ C\Delta t \sum_{k=0}^{M}\left\{D({\theta}_{\mathbf{A}}^{k},{\theta}_{\mathbf{A}}^{k})
+\|\partial {\theta}_{\mathbf{A}}^{k}\|_{\mathbf{L}^2}^{2}+\|\theta_{\Psi}^{k}\|_{\mathcal{L}^2}^{2}
+\|\nabla\theta_{\Psi}^{k}\|_{\mathbf{L}^2}^{2}\right\}.}
\end{array}
\end{equation}

By applying the discrete Gronwall's inequality, we have
\begin{equation}\label{eq:4-82}
 \max_{1\leq k \leq M}\left\{\|\theta_{\Psi}^{k}\|_{\mathcal{L}^2}^{2}
+\|\nabla\theta_{\Psi}^{k}\|_{\mathbf{L}^2}^{2}
+\|\partial \theta_{\mathbf{A}}^{k}\|_{\mathbf{L}^2}^{2}
+D({\theta}_{\mathbf{A}}^{k}, {\theta}_{\mathbf{A}}^{k})\right\} \leq  C\left\{h^{2r}+(\Delta t)^{4}\right\}. 
\end{equation}

Combine (\ref{eq:4-82}) with the interpolation error estimates (\ref{eq:4-1}) and we can complete the proof of Theorem~\ref{thm2-1}.

\section{Numerical tests}\label{sec-6}

To validate the developed algorithm and to confirm the theoretical analysis reported
in this paper, we present numerical simulations for the following case studies.

\begin{exam}\label{exam6-1}
We consider the Maxwell-Schr\"{o}dinger system (\ref{eq:1-2}), where the initial-boundary conditions are as
follows:
\begin{equation*}
\begin{array}{@{}l@{}}
{\displaystyle \Psi(\mathbf{x},t)=0,\quad \mathbf{A}(\mathbf{x},t)\times\mathbf{n}=0,  \quad (\mathbf{x},t)\in
\partial \Omega\times(0,T),} \\[2mm]
{\displaystyle \Psi(\mathbf{x},0) = 2\sqrt{2}\sin(\pi x_1)\sin(\pi x_2)\sin(\pi x_3) , \quad
\mathbf{A}_{t}(\mathbf{x},0)=\mathbf{A}_1(\mathbf{x})=0,  }\\[2mm]
{\mathbf{A}(\mathbf{x},0)=\mathbf{A}_0(\mathbf{x})
=\big(10x_1 x_2 x_3(1-x_2)(1-x_3),  10x_1 x_2 x_3(1-x_1)(1-x_3), }\\[2mm]
{\displaystyle\qquad \qquad \qquad \qquad 10x_1 x_2 x_3(1-x_1)(1-x_2)\big). }
\end{array}
\end{equation*}
Here we take $\Omega=(0,1)^{3} $, $ V_0 = 0 $, $ T = 0.5$, and the time step $\Delta t = 0.0025 $. Note that the initial wave function $\Psi(\mathbf{x},0)$ is the eigenfunction of the stationary Schr\"{o}dinger's equation.
\end{exam}
The numerical results are displayed in Fig.~5.1.

\begin{figure}
\begin{center}
\tiny{a}\includegraphics[width=6cm,height=6cm] {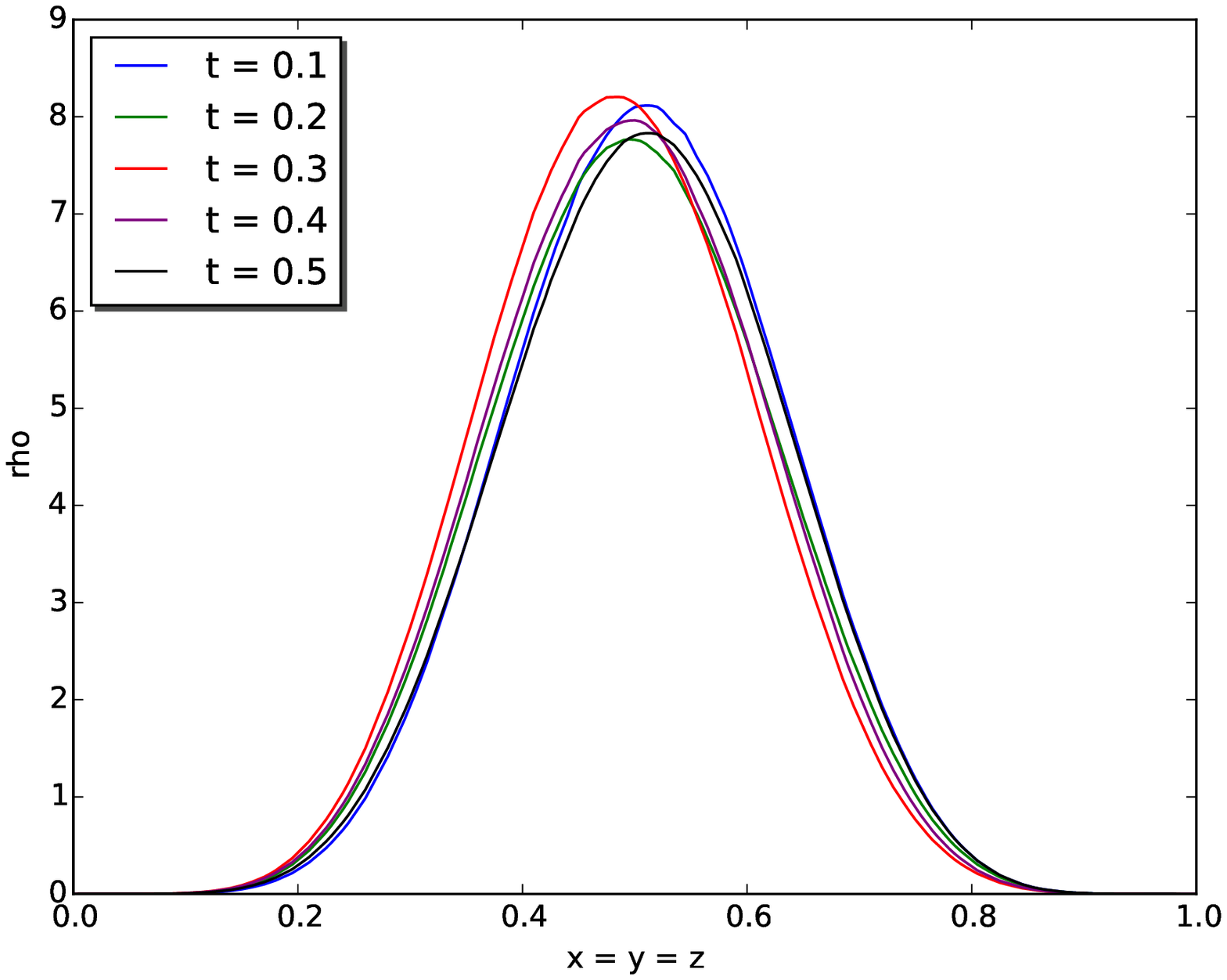}
\tiny{b}\includegraphics[width=6cm,height=6cm] {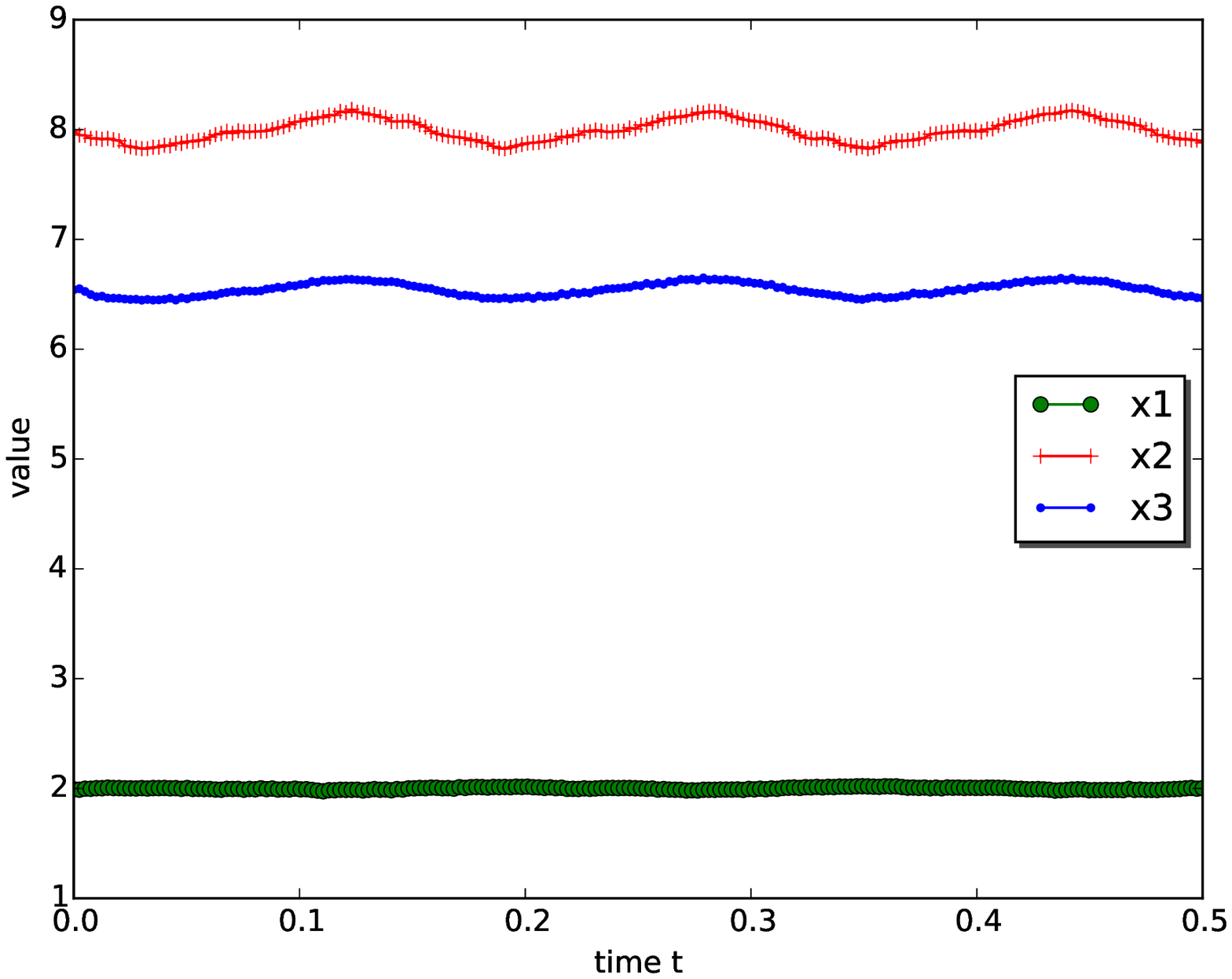}
\caption{Example~\ref{exam6-1} : (a) The evolution of the density function $ \rho(\mathbf{x},t) $ on the line  $
x_1=x_2=x_3 $ at time $t =0.1, 0.2, 0.3 ,0.4, 0.5$; (b) the evolution of $ \rho(\mathbf{x}^{(1)},t) $, $ \rho(\mathbf{x}^{(2)},t) $ and $ \rho(\mathbf{x}^{(3)}, t)$, where $ \mathbf{x}^{(1)}=(0.25,0.5,0.75) $, $\mathbf{x}^{(2)}=(0.5,0.5,0.5) $ and $ \mathbf{x}^{(3)}=(0.4,0.5,0.6)$.
}
\end{center}
\end{figure}\label{fig:6-1}

\begin{rem}\label{rem6-1}
The numerical results illustrated in Fig.~5.1 clearly show that the change of $\rho(\mathbf{x}, t)$ is smooth with respect to $t$ and the assumption on which the modified Maxwell-Schr\"{o}dinger equations are based is valid in this case.
\end{rem}

\begin{exam}\label{exam6-2}
We consider the modified Maxwell--Schr\"{o}dinger's equations as follows:
\begin{equation}\label{eq:6-11}
\left\{
\begin{array}{@{}l@{}}
{\displaystyle  -\mathrm{i}\frac{\partial \Psi}{\partial t}+
\frac{1}{2}\left(\mathrm{i}\nabla +\mathbf{A}\right)^{2}\Psi
 + V_{0}\Psi = f ,\,\, (\mathbf{x},t)\in
\Omega\times(0,T),}\\[2mm]
{\displaystyle \frac{\partial ^{2}\mathbf{A}}{\partial t^{2}}+\nabla\times
(\nabla\times \mathbf{A}) - \gamma\nabla(\nabla \cdot \mathbf{A}) +\frac{\mathrm{i}}{2}\big(\Psi^{*}\nabla{\Psi}-\Psi\nabla{\Psi}^{*}\big)}\\[2mm]
{\displaystyle \quad\quad+\,\,\vert\Psi\vert^{2}\mathbf{A} =\mathbf{g}(\mathbf{x},t),\quad \,\, (\mathbf{x},t)\in \Omega\times(0,T),}
\end{array}
\right.
\end{equation}
where the initial-boundary conditions are given in (\ref{eq:1-7})-(\ref{eq:1-8}).
We take $\Omega = (0,1)^3$, $T = 4$, $ \gamma = 1$ and $V_0 = 5$. The exact solution $(\Psi,\mathbf{A})$ of (\ref{eq:6-11}) is given by
\begin{equation*}
\begin{array}{@{}l@{}}
{\displaystyle \Psi(\mathbf{x}, t) = 20 e^{\mathrm{i}t}(1+3t^2)\exp\big((x_1+x_2+x_3)/5\big)x_1 x_2 x_3(1-x_1)(1-x_2)(1-x_3) }\\[2mm]
{\displaystyle   \qquad\qquad\qquad+  5.0e^{\mathrm{i}\pi t}  \sin(2\pi x_1)\sin(2\pi x_2)\sin(2\pi x_3), }
\end{array}
\end{equation*}
\begin{equation*}
\begin{array}{@{}l@{}}
{\displaystyle \mathbf{A}(\mathbf{x},t)=\sin(\pi t)\Big(\cos(2\pi x_1)\sin(2\pi x_2)\sin(2\pi x_3),
\sin(2\pi x_1)\cos(2\pi x_2)\sin(2\pi x_3),}\\[2mm]
{\displaystyle \quad \sin(2\pi x_1)\sin(2\pi x_2)\cos(2\pi x_3)\Big)
+ \cos(\pi t)\Big(\cos(\pi x_1)\sin(\pi x_2)\sin(\pi x_3),}\\[2mm]
{\displaystyle \quad
\sin(\pi x_1)\cos(\pi x_2)\sin(\pi x_3), \sin(\pi x_1)\sin(\pi x_2)\cos(\pi x_3)\Big).}
\end{array}
\end{equation*}
The functions $f$ and $\mathbf{g}$ in (\ref{eq:6-11}) are chosen correspondingly to the exact solution $(\Psi,\mathbf{A})$.
\end{exam}

A uniform tetrahedral partition is generated with $M+1$ nodes in each direction and $6M^{3}$ elements in total.
We solve the system(\ref{eq:6-11}) by the proposed Crank-Nicolson Galerkin finite element scheme (\ref{eq:2-9}) with linear elements and quadratic elements, respectively. To confirm our error analysis, we take $\Delta t = h^{\frac{1}{2}}$ for the linear element method and $\Delta t = h$ for the quadratic element method respectively. Numerical results for the linear element method and the quadratic element method at time $t=1.0, 2.0, 3.0, 4.0 $ are listed in Tables~\ref{table6-1} and~\ref{table6-2}, respectively.

\begin{table}[htb]
\caption{$H^1$ error of linear FEM with $h = \frac{1}{M}$ and $\Delta t = h^{\frac12}$.}\label{table6-1}
\begin{center}
\begin{tabular}{c|cccc}
   \hline\hline
    & &   $\Vert \mathbf{A}^{k}_{h} - \mathbf{A}(\cdot, t_k) \Vert _{\mathbf{H}^{1}} $ \\
   \hline
    t &  M=25 & M=50 & M=100 & Order   \\
    \hline
    1.0 & 9.7908e-01 & 4.8501e-01 & 2.1082e-01 & 1.11 \\
    \hline
     2.0 & 7.6414e-01 & 3.7807e-01 & 1.7681e-01 & 1.06 \\
     \hline
     3.0 & 6.3094e-01 & 3.1006e-01 & 1.5308e-01 & 1.02 \\
    \hline
     4.0 & 7.2739e-01 & 3.5204e-01 & 1.7705e-01 & 1.02 \\
     \hline\hline
   & &   $ \Vert \Psi_{h}^{k}-\Psi(\cdot,t_k)\Vert_{\mathcal{H}^1}$    \\
     \hline
    t &  M=25 & M=50 & M=100 & Order   \\
    \hline
    1.0 & 6.8289e-01 & 3.3004e-01 & 1.5046e-01 & 1.09 \\
    \hline
     2.0 & 8.0035e-01 & 3.6227e-01 & 1.6032e-01 & 1.16 \\
     \hline
     3.0 & 4.1192e-01 & 1.9485e-01 & 1.0286e-01 & 1.00 \\
    \hline
     4.0 & 2.3418e-01 & 1.1430e-01 & 5.6022e-02  & 1.03 \\
     \hline\hline

  \end{tabular}
  \end{center}
\end{table}

\begin{table}[htb]
\caption{$H^1$ error of quadratic FEM with $h = \Delta t= \frac{1}{M}.$}\label{table6-2}
\begin{center}
\begin{tabular}{c|cccc}
   \hline\hline
    & &   $\Vert \mathbf{A}^{k}_{h} - \mathbf{A}(\cdot, t_k) \Vert _{\mathbf{H}^{1}} $ \\
   \hline
    t &  M=25 & M=50 & M=100 & Order   \\
    \hline
    1.0 & 3.3770e-02 & 8.4984e-03 & 2.2743e-03 & 1.95 \\
    \hline
     2.0 & 2.2786e-02 & 5.7068e-03  & 1.4966e-03 & 1.96 \\
     \hline
     3.0 & 3.4016e-02 & 8.9115e-03 & 2.4360e-03 & 1.90 \\
    \hline
     4.0 & 3.9787e-02 &  9.0740e-03 & 2.3467e-03  & 2.04 \\
     \hline\hline
   & &   $ \Vert \Psi_{h}^{k}-\Psi(\cdot,t_k)\Vert_{\mathcal{H}^1}$    \\
     \hline
    t &  M=25 & M=50 & M=100 & Order   \\
    \hline
    1.0 & 2.8221e-02 & 6.8364e-03 & 1.8024e-03 & 1.98 \\
    \hline
     2.0 & 4.5738e-02 & 1.1608e-02 & 2.6719e-03 & 2.05 \\
     \hline
     3.0 & 3.2712e-02 & 8.5698e-03 & 2.2101e-03 & 1.94 \\
    \hline
     4.0 & 2.1868e-02 & 5.1495e-03 & 1.3250e-03  & 2.02 \\
     \hline\hline

  \end{tabular}
  \end{center}
\end{table}

\begin{rem}\label{rem6-2}
Numerical results in Tables~\ref{table6-1} and~\ref{table6-2} are in good agreement with the theoretical analysis,
see Theorem~\ref{thm2-1}.
\end{rem}

\begin{exam}\label{exam6-3}
We consider the following modified Maxwell--Schr\"{o}dinger's equations 
\begin{equation} \label{eq:6-12}
\left\{
\begin{array}{@{}l@{}}
{\displaystyle  -\mathrm{i}\frac{\partial \Psi}{\partial t}+
\frac{1}{2}\left(\mathrm{i}\nabla +\mathbf{A}\right)^{2}\Psi
 + V_{0}\Psi = 0 ,\,\, (\mathbf{x},t)\in
\Omega\times(0,T),}\\[2mm]
{\displaystyle \frac{\partial ^{2}\mathbf{A}}{\partial t^{2}}+\nabla\times
(\nabla\times \mathbf{A}) - \gamma\nabla(\nabla \cdot \mathbf{A}) +\frac{\mathrm{i}}{2}\big(\Psi^{*}\nabla{\Psi}-\Psi\nabla{\Psi}^{*}\big)}\\[2mm]
{\displaystyle \quad\quad+\,\,\vert\Psi\vert^{2}\mathbf{A} =\mathbf{g}(\mathbf{x},t),\quad \,\, (\mathbf{x},t)\in \Omega\times(0,T).}\\[2mm]
{\displaystyle \Psi(\mathbf{x},t)=0,\quad \mathbf{A}(\mathbf{x},t)\times\mathbf{n}=0, \quad \nabla \cdot \mathbf{A}(\mathbf{x},t) = 0, \quad (\mathbf{x},t)\in
\partial \Omega\times(0,T),}\\[2mm]
{\displaystyle \Psi(\mathbf{x},0) = \Psi_0(\mathbf{x}),\quad\mathbf{A}(\mathbf{x},0)=\mathbf{A}_{0}(\mathbf{x}),\quad
\mathbf{A}_{t}(\mathbf{x},0)=\mathbf{A}_{1}(\mathbf{x}),}
\end{array}
\right.
\end{equation}
with
\begin{equation}
\begin{array}{@{}l@{}}
{\displaystyle \Psi_{0}(\mathbf{x}) = 2\sqrt{2}\sin(\pi x_1)\sin(\pi x_2)\sin(\pi x_3) , \quad
\mathbf{A}_{0}(\mathbf{x})=\mathbf{A}_{1}(\mathbf{x})=0,} \\[2mm]
{\displaystyle \mathbf{g}(\mathbf{x}, t) = \left(10\sin(1.5{\pi}^{2} t), 10\sin(1.5{\pi}^{2} t), 10\cos(1.5{\pi}^{2} t)\right).}
\end{array}
\end{equation}

In this example we take $\Omega = (0,1)^3$, $T = 10.0$ , $\gamma = 1$, $V_0 = 0$. Using the mesh in Example~\ref{exam6-2} with M = 50, we solve the system (\ref{eq:6-12}) by the proposed Crank-Nicolson Galerkin finite element scheme (\ref{eq:2-9}) with linear elements. The time step $\Delta t = 0.0025 $. 

In Fig.~5.2 we display the numerical results of $\rho = \vert\Psi\vert^{2}$  on the line  $
x_1=x_2=x_3 $  and on the intersection $x_1 = 0.5$ at time $t =2.0, 4.0, 6.0 ,8.0, 10.0$.

In Fig.~5.3 we plot the numerical results of $\rho = \vert\Psi\vert^{2}$  on the intersections $x_1 = 0.4, \,x_2=0.4 \,$ and $ x_3 = 0.4$ at time $t =10.0$, respectively.

\begin{figure}
\begin{center}
\tiny{(a)}\includegraphics[width=6cm,height=6cm] {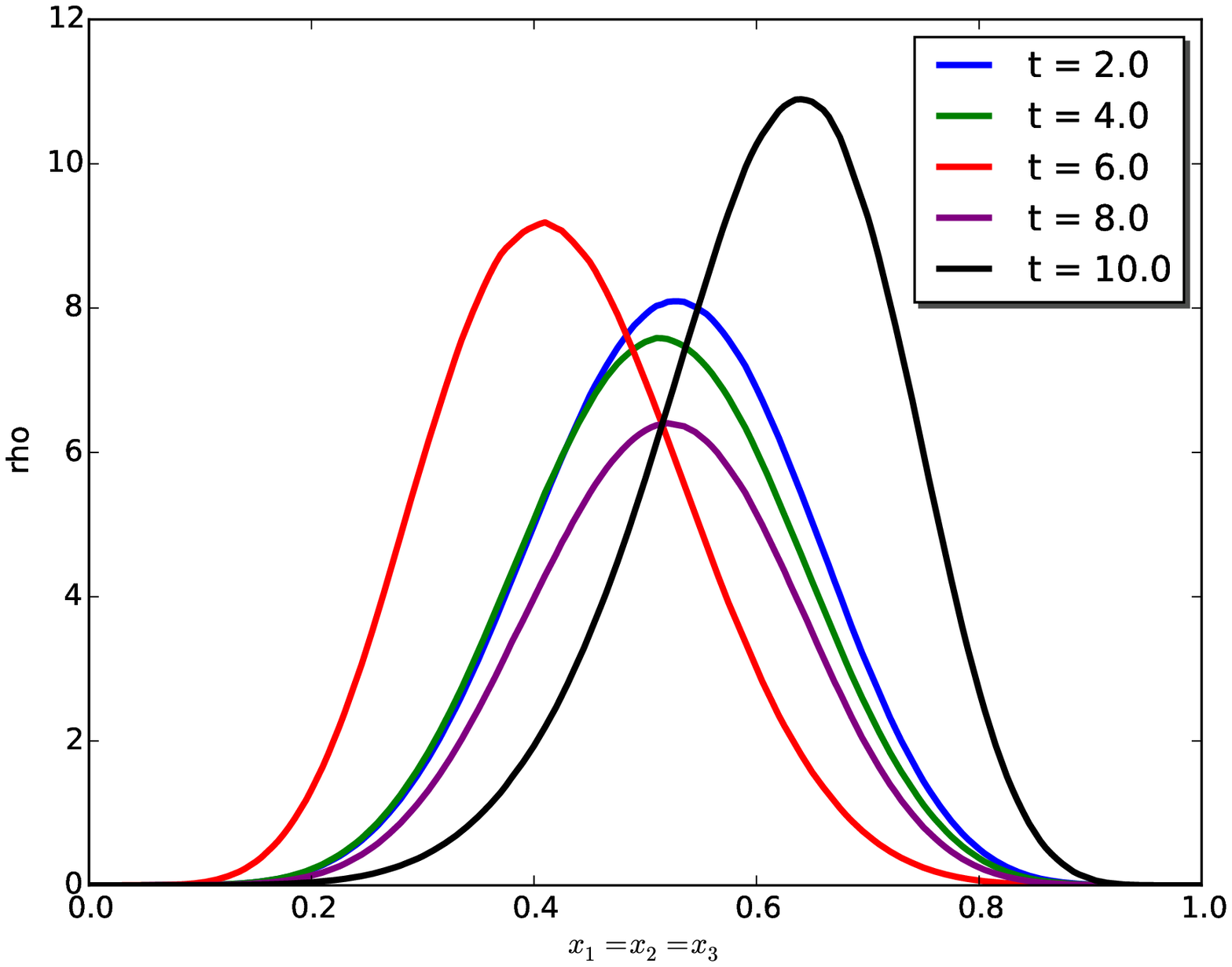}
\tiny{(b)}\includegraphics[width=6cm,height=6cm] {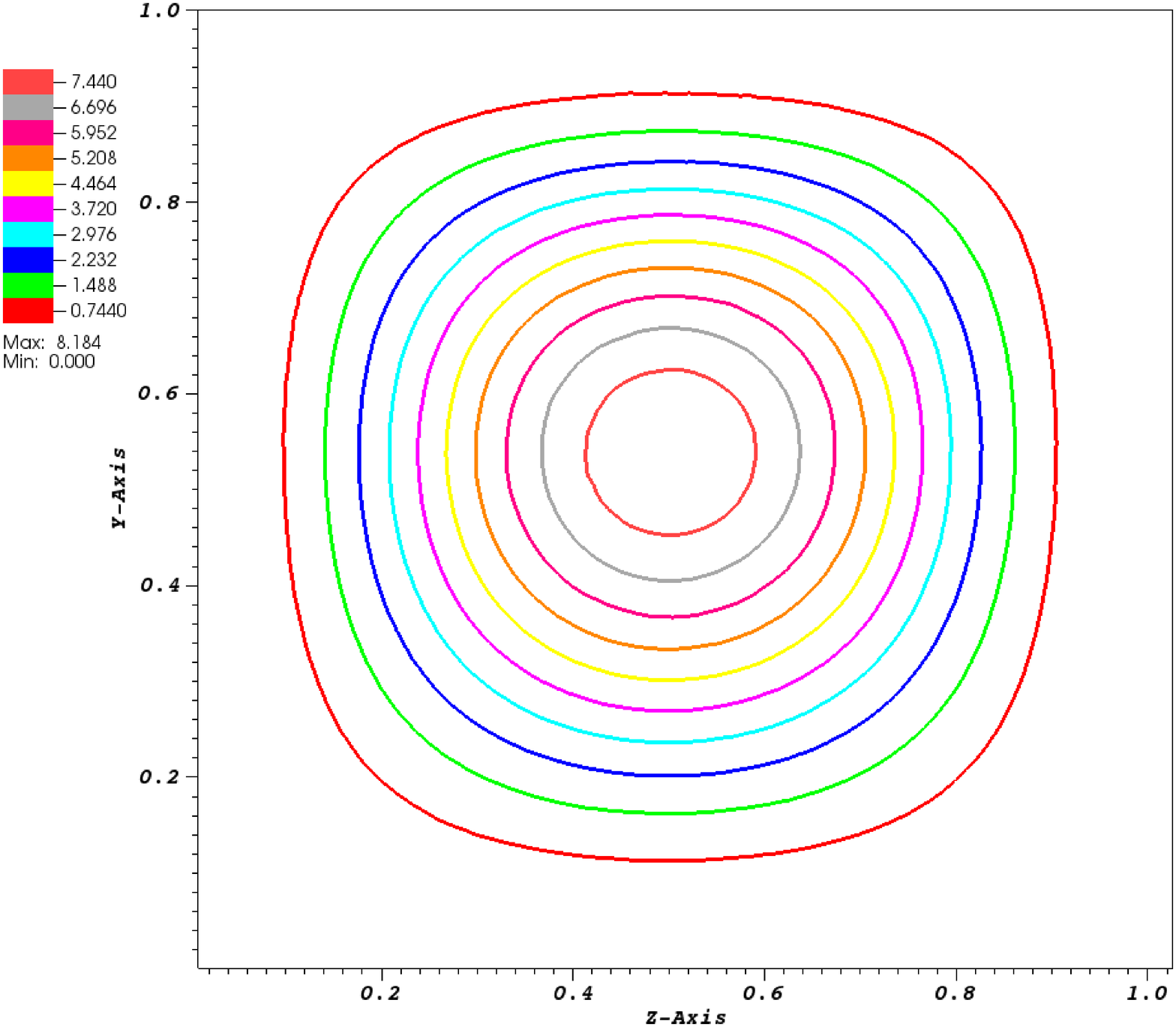}
\tiny{(c)}\includegraphics[width=6cm,height=6cm] {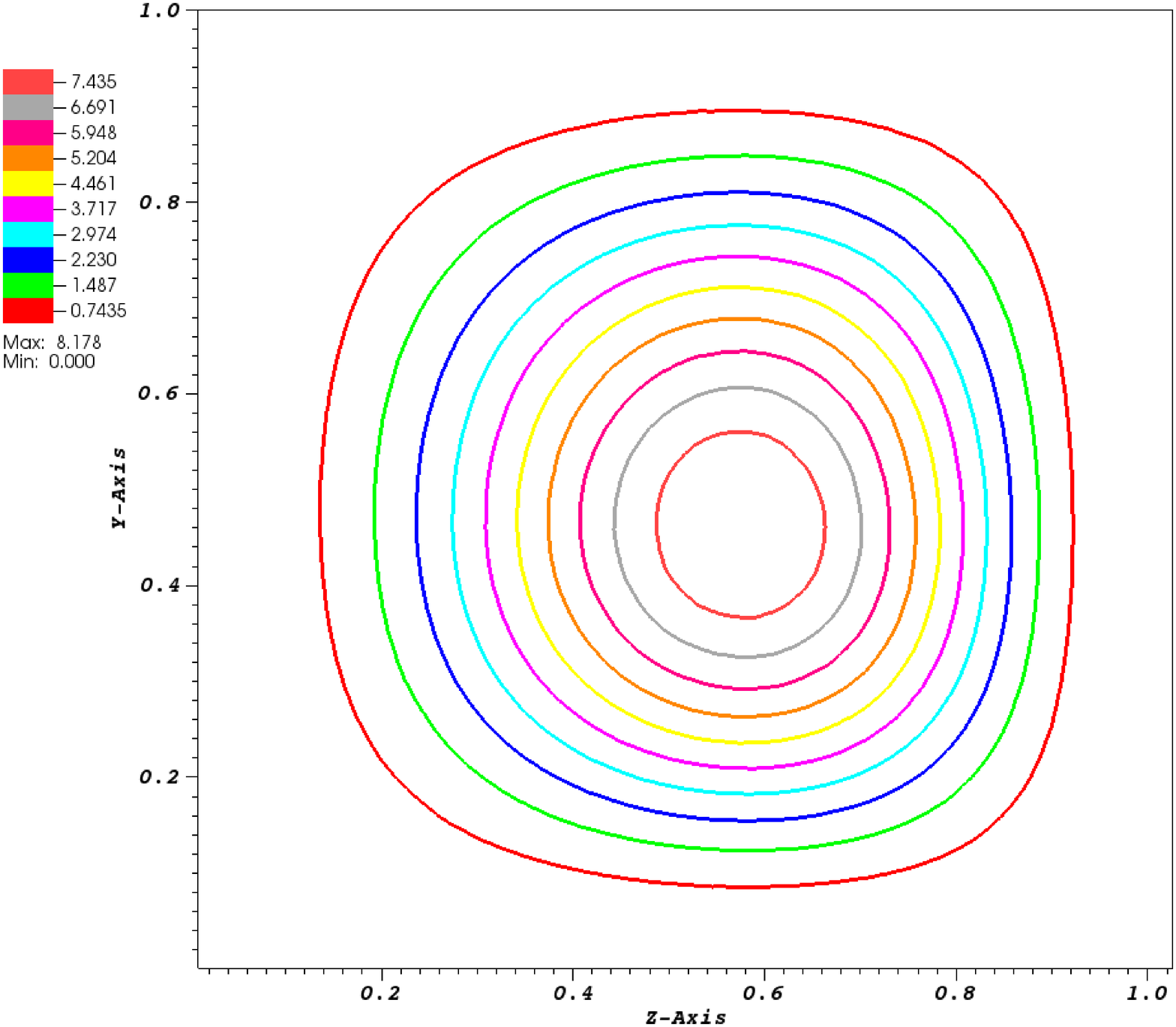}
\tiny{(d)}\includegraphics[width=6cm,height=6cm] {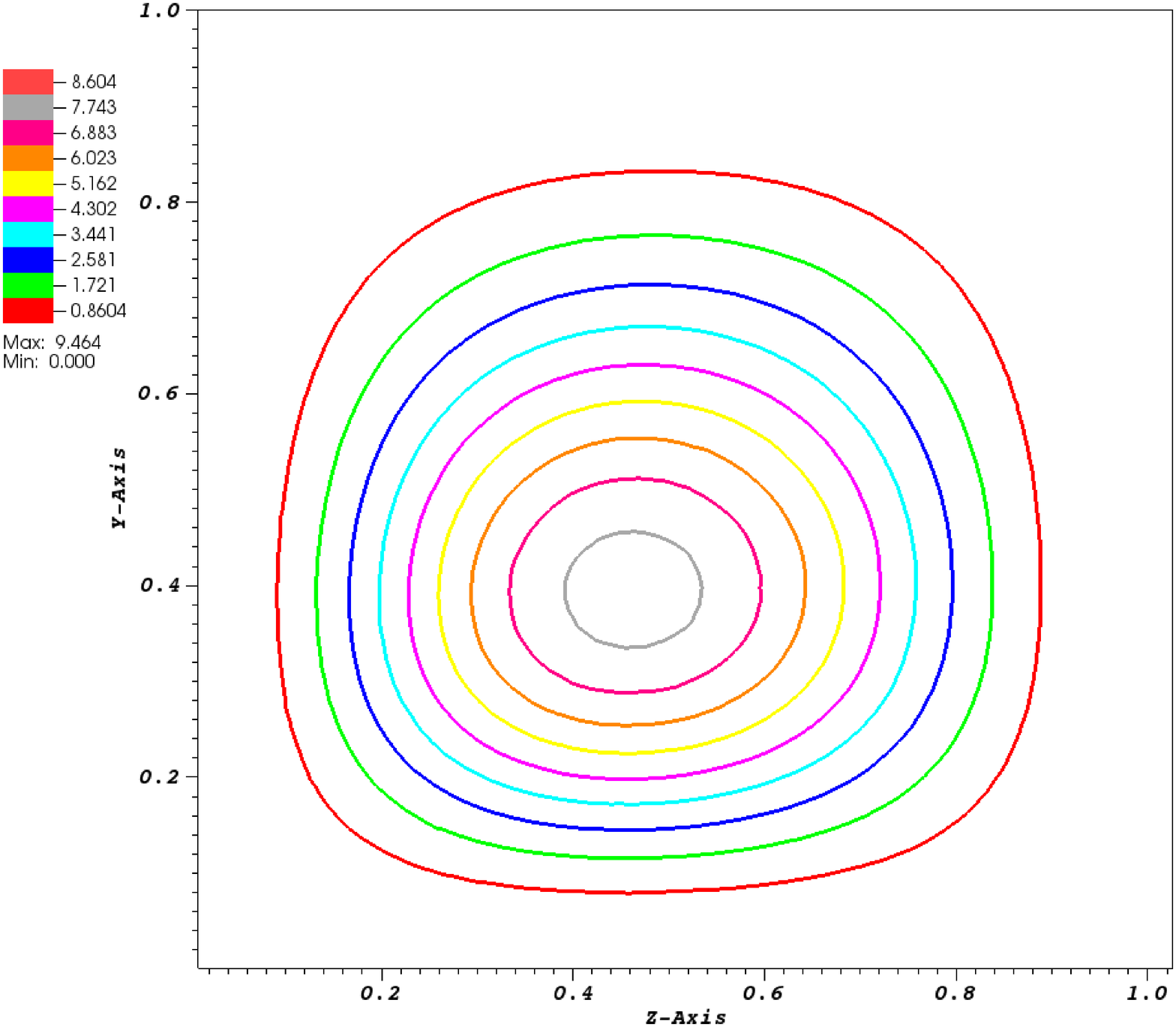}
\tiny{(e)}\includegraphics[width=6cm,height=6cm] {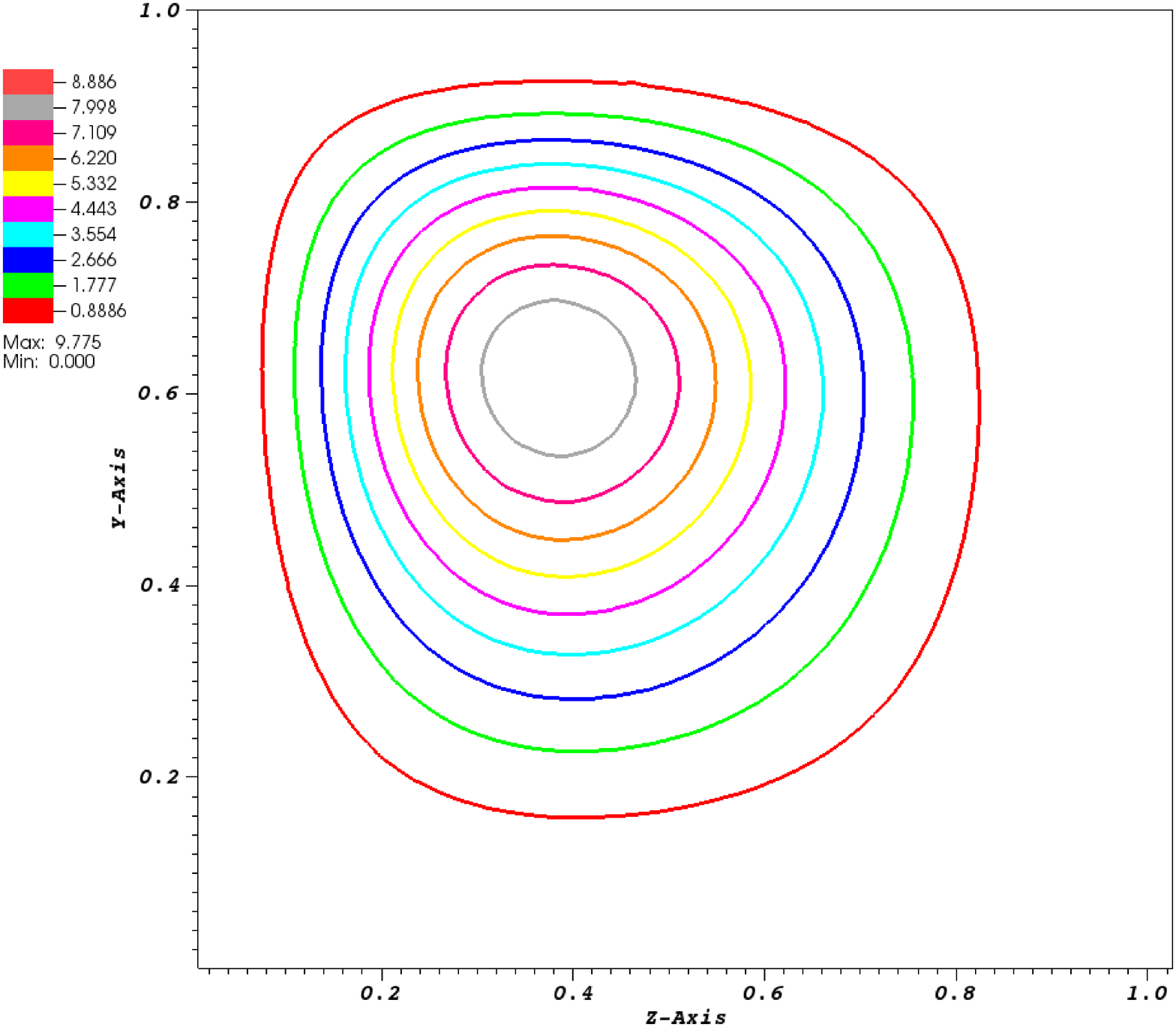}
\tiny{(f)}\includegraphics[width=6cm,height=6cm] {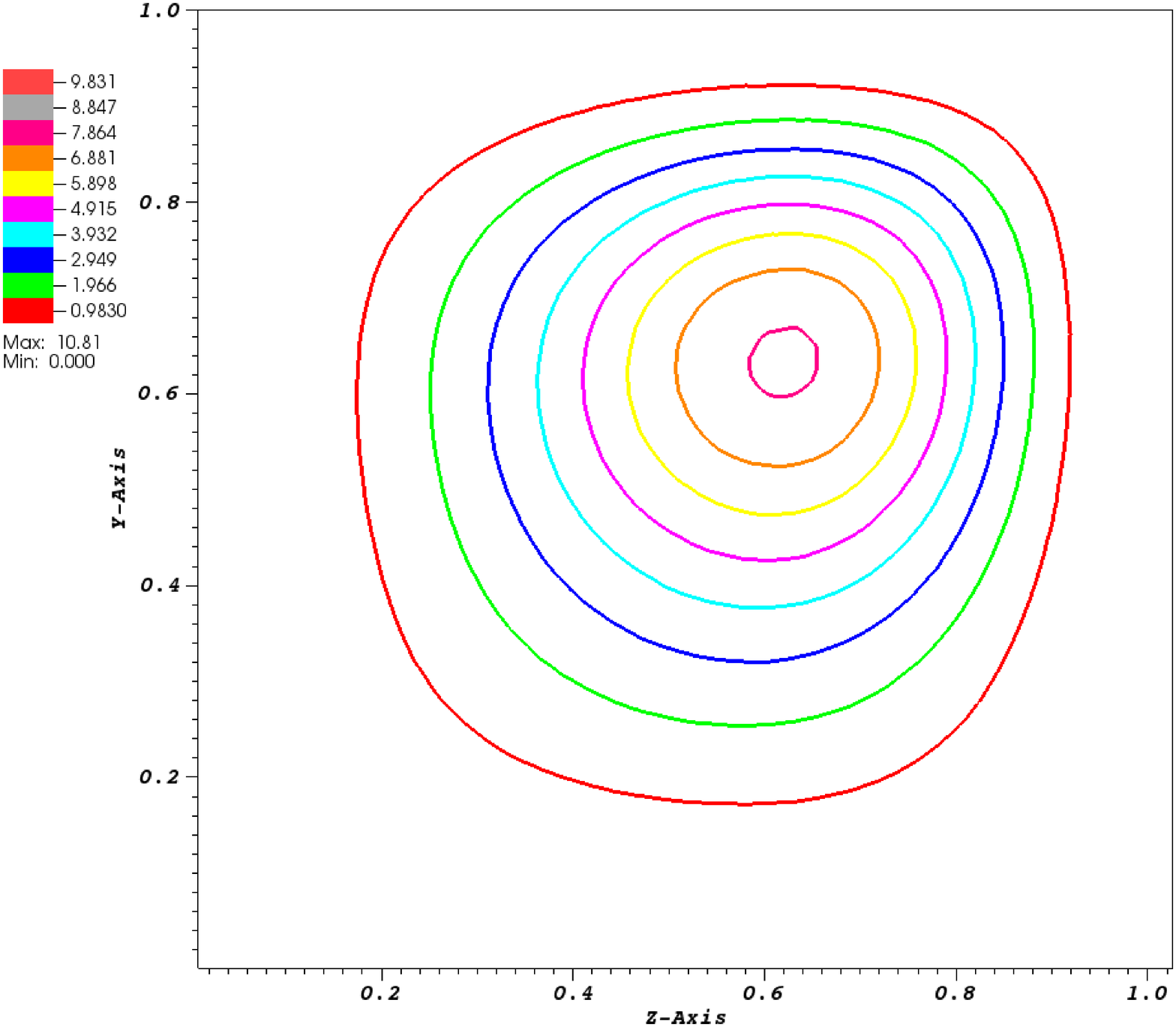}
\caption{Example~\ref{exam6-3}: numerical results of the density function  $\rho = \vert\Psi\vert^{2}$ on the line  $
x_1=x_2=x_3 $ and  the contour plots of $\rho = \vert\Psi\vert^{2}$  on the intersection $x_1 = 0.5$ at time $t =2.0, 4.0, 6.0 ,8.0, 10.0$.
(a)  $\rho = \vert\Psi\vert^{2}$ on the line  $
x_1=x_2=x_3 $ at time $t =2.0, 4.0, 6.0 ,8.0, 10.0$;\, (b) the contour plot at time $t =2.0$;\, (c) the contour plot  at time $t =4.0$; \, (d) the contour plot   at time $t =6.0$; \,(e) the contour plot  at time $t =8.0$;\,
(f) the contour plot  at time $t =10.0$;\, }\label{fig6.2}
\end{center}
\end{figure}

\begin{figure}
\begin{center}
\tiny{(a)}\includegraphics[width=6cm,height=6cm] {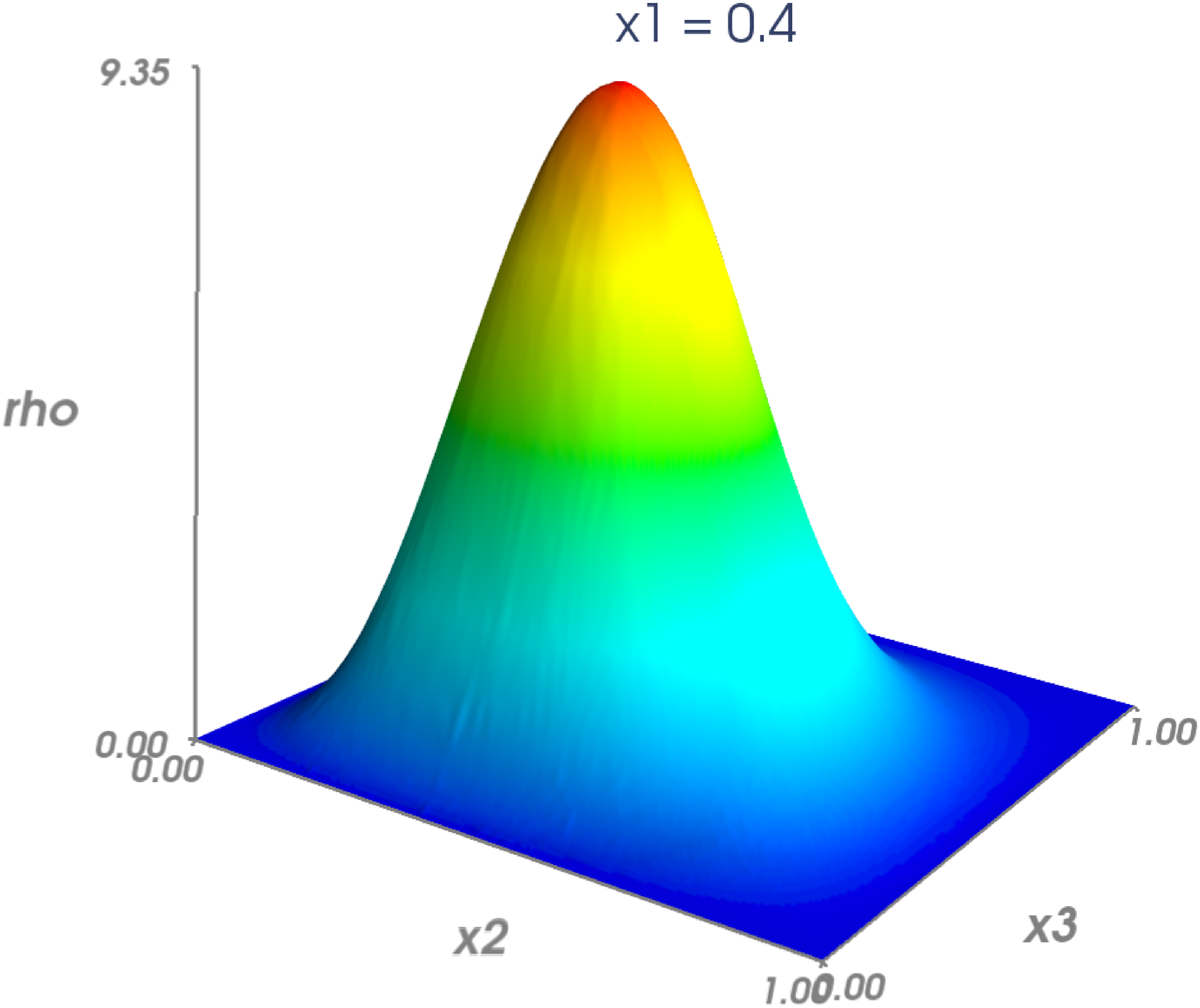}
\tiny{(b)}\includegraphics[width=6cm,height=6cm] {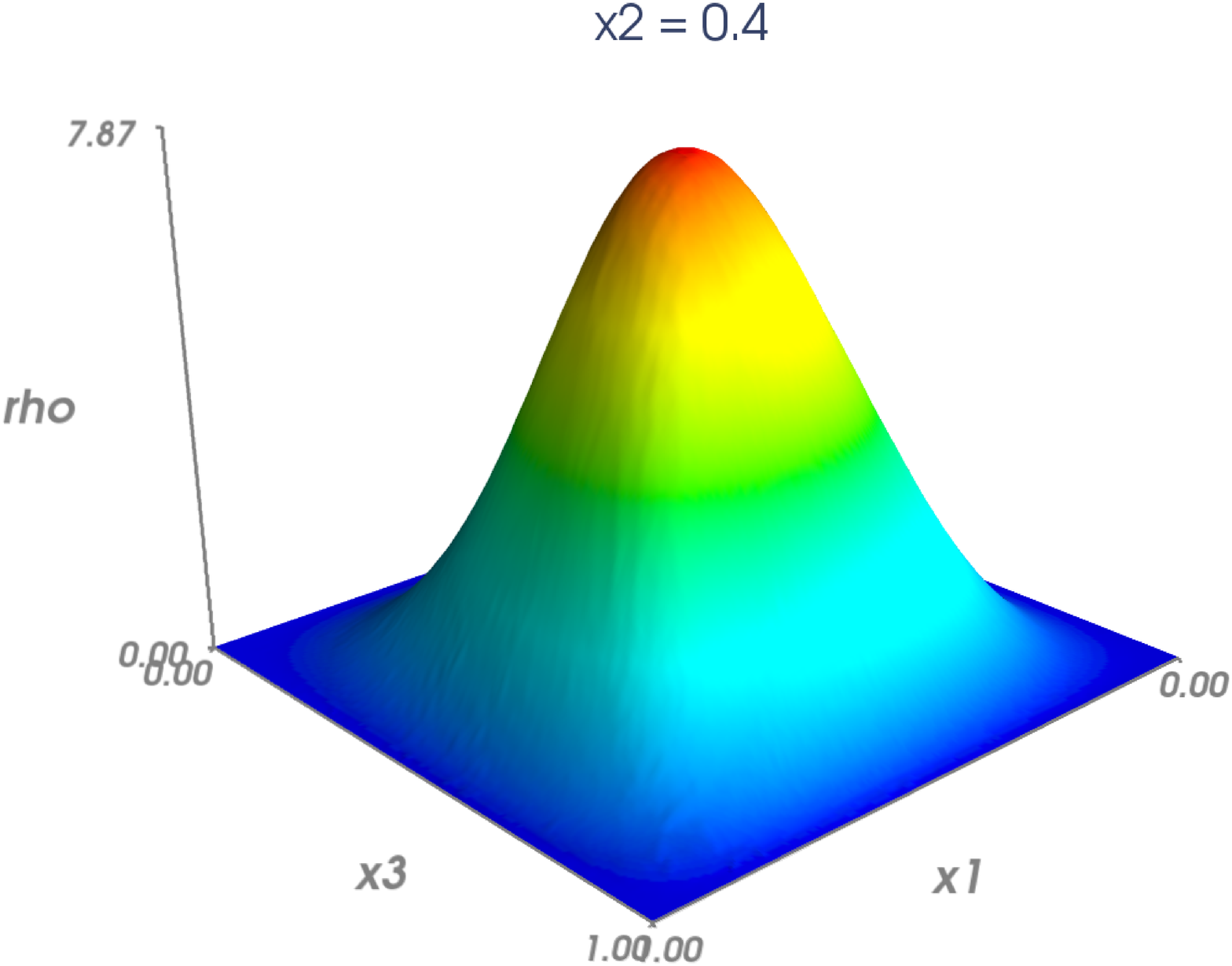}
\tiny{(c)}\includegraphics[width=6cm,height=6cm] {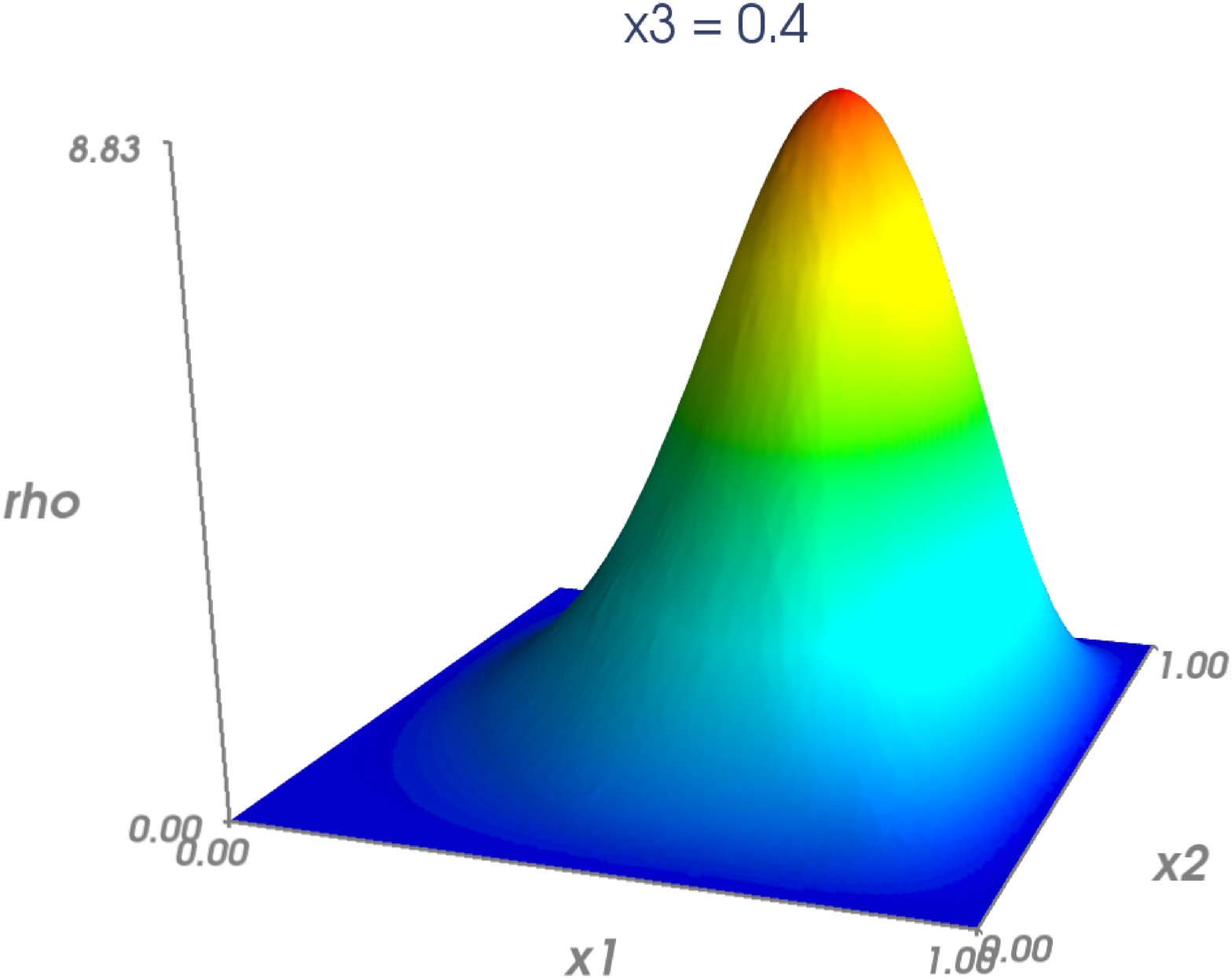}
\caption{Example~\ref{exam6-3}: numerical results of the density function  $\rho = \vert\Psi\vert^{2}$ on the intersection $x_1 = 0.4,\,x_2 = 0.4$ and $x_3 = 0.4$ at time $t = 10.0$.
(a)  $\rho = \vert\Psi\vert^{2}$ on the intersection $x_1 = 0.4$ ;\, (b) $\rho = \vert\Psi\vert^{2}$ on the intersection $x_2 = 0.4$;\, (c) $\rho = \vert\Psi\vert^{2}$ on the intersection $x_3 = 0.4$; \, }\label{fig6.3}
\end{center}
\end{figure}

Althouth the Maxwell-Schr\"{o}dinger system and the time-dependent Ginzburg-Landau equations are somehow formally similar, they describe the different physical phenomenons. The time-dependent Ginzburg-Landau equations describe the vortex dynamics of superconductor \cite{Gao-1, Gao-2, Gao-3} while the Maxwell-Schr\"{o}dinger equations describe the wave packet dynamics of an electron. As can be seen in Fig.~5.2-5.3, the wave packet of the electron is located at the center of the computational domain at first and the external and its self-induced electromagnetic fields cause the motion of the wave packet.  
Unlike the time-dependent Ginzburg-Landau equations, no stable state is observed in our computation.

\end{exam}

\section{Conclusions}
We have presented the optimal $H^1$ error estimates of a Crank-Nicolson Galerkin finite element method 
for the modified Maxwell-Schr\"{o}dinger equations , which are derived from the original equations under some assumptions. The techniques used in this paper may also be applied to other nonlinear PDEs, such as the Ginzburg-Landau equations. The original Maxwell-Schr\"{o}dinger system is challenging and difficult to perform numerical computation and theoretical analysis. Our work can serve as an elementary attempt for the numerical analysis of this system. We will study the original system in a further work using the mixed finite element method.

\end{document}